\documentclass{article}
\usepackage{a4wide,eucal,enumerate,mathrsfs,xspace}
\usepackage{amsmath,amssymb,epsfig,bbm}
\usepackage[toc,page]{appendix}
\usepackage[usenames]{color}
\usepackage{hyperref}
\usepackage{cite}
\newcommand{\N}{\mathbb{N}}

\newcommand{\R}{\mathbb{R}}

\renewcommand{\AA}{\mathscr{A}}

\newcommand{\PP}{\mathscr{P}}

\newcommand{\UU}{\mathscr{U}}

\newcommand{\calC}{{\ensuremath{\mathcal C}}}

\newcommand{\cE}{{\ensuremath{\mathcal E}}}
\newcommand{\cF}{{\ensuremath{\mathcal F}}}
\newcommand{\cG}{{\ensuremath{\mathcal G}}}
\newcommand{\cH}{{\ensuremath{\mathcal H}}}
\newcommand{\cK}{{\ensuremath{\mathcal K}}}

\newcommand{\cL}{{\ensuremath{\mathcal L}}}
\newcommand{\cM}{{\ensuremath{\mathcal M}}}

\newcommand{\cP}{{\ensuremath{\mathcal P}}}

\newcommand{\cS}{{\ensuremath{\mathcal S}}}

\newcommand{\cW}{{\ensuremath{\mathcal W}}}

\newcommand{\ff}{{\mbox{\boldmath$f$}}}

\newcommand{\oo}{{\mbox{\boldmath$o$}}}
\newcommand{\pp}{{\mbox{\boldmath$p$}}}
\newcommand{\qq}{{\mbox{\boldmath$q$}}}

\newcommand{\uu}{{\mbox{\boldmath$u$}}}
\newcommand{\vv}{{\mbox{\boldmath$v$}}}
\newcommand{\ww}{{\mbox{\boldmath$w$}}}
\newcommand{\xx}{{\mbox{\boldmath$x$}}}
\newcommand{\yy}{{\mbox{\boldmath$y$}}}

\newcommand{\suu}{{\mbox{\scriptsize\boldmath$u$}}}

\newcommand{\sxx}{{\mbox{\scriptsize\boldmath$x$}}}

\newcommand{\fF}{{\mbox{\boldmath$F$}}}
\newcommand{\gG}{{\mbox{\boldmath$G$}}}

\newcommand{\nnu}{{\mbox{\boldmath$\nu$}}}

\newcommand{\xxi}{{\mbox{\boldmath$ \xi$}}}

\newcommand{\snnu}{{\mbox{\scriptsize\boldmath$\nu$}}}

\newcommand{\sfd}{{\sf d}}

\newcommand{\frF}{{\frak F}}

\newcommand{\Kliminf}{K\kern-3pt-\kern-2pt\mathop{\rm
lim\,inf}\limits}  
\newcommand{\Klimsup}{K\kern-3pt-\kern-2pt\mathop{\rm lim\,sup}\limits}  
\newcommand{\supp}{\mathop{\rm supp}\nolimits}   

\newcommand{\argmin}{\mathop{\rm argmin}\limits}   
\newcommand{\Lip}{\mathop{\rm Lip}\nolimits}          
\renewcommand{\d}{{\mathrm d}}
\newcommand{\dt}{{\d t}}

\newcommand{\restr}[1]{\lower3pt\hbox{$|_{#1}$}}

\newcommand{\la}{{\langle}}                  
\newcommand{\ra}{{\rangle}}

\newcommand{\eps}{\varepsilon}  
\newcommand{\nchi}{{\raise.3ex\hbox{$\chi$}}}
\newcommand{\Rd}{{\R^d}}

\newcommand{\media}{\mkern12mu\hbox{\vrule height4pt           %
          depth-3.2pt                                 
          width5pt}\mkern-16.5mu\int\nolimits}        

\def\qed{\ifmmode 
  \else \leavevmode\unskip\penalty9999 \hbox{}\nobreak\hfill
  \fi               
    \qquad           \hbox{\hskip.5em $\square$
                \hskip.1em}}
\def\endproofsym{\qed}

\newenvironment{proof}[1][Proof]{\def\endproofsym{\qed}\trivlist\item[\hskip\labelsep{%
\noindent{\normalfont\emph{#1}.}\hskip .321429\parindent}]\ignorespaces}
{\endproofsym\endtrivlist}

\newcommand{\nc}{\normalcolor}

\newcommand{\GGG}{\color{blue}}

\usepackage{ifthen}
\numberwithin{equation}{section}

\newtheorem{theorem}{Theorem}[section]

\newtheorem{lemma}[theorem]{Lemma}
\newtheorem{proposition}[theorem]{Proposition}
\newtheorem{definition}[theorem]{Definition}

\newtheorem{remark}[theorem]{Remark}

\renewcommand{\GGG}{}
\newcommand{\mres}{\mathbin{\vrule height 1.6ex depth 0pt width
0.13ex\vrule height 0.13ex depth 0pt width 1.3ex}}

\renewcommand{\UU}{U}
\newcommand{\MF}{}

\newenvironment{system} 
{\left\lbrace\begin{array}{@{}l@{}}}
{\end{array}\right.}

\title{Mean-field optimal control as \\
  Gamma-limit of finite agent controls}
\author{Massimo Fornasier \thanks{Department of Mathematics, TU M\"unchen. Boltzmannstr. 3, Garching bei M\"unchen, D-85748, Germany.  \texttt{ massimo.fornasier@ma.tum.de}} \and 
Stefano Lisini \thanks{Dipartimento di Matematica ''F. Casorati'', Universit\`a di Pavia. Via Ferrata 5, 27100 Pavia, Italy. \texttt{stefano.lisini@unipv.it}} \and 
Carlo Orrieri \thanks{Dipartimento di Matematica ''G. Castelnuovo'', Sapienza Universit\`a di Roma. Piazzale Aldo Moro 5, 00185 Roma, Italy. \texttt{orrieri@mat.uniroma1.it}} \and 
Giuseppe Savar\'e \thanks{Dipartimento di Matematica ''F. Casorati'', Universit\`a di Pavia. Via Ferrata 5, 27100 Pavia, Italy. \texttt{giuseppe.savare@unipv.it}}}

\begin{document}

\maketitle

\begin{abstract}
{\MF This paper focuses on the role of a government of a large population of interacting agents as a mean field optimal control problem derived from deterministic finite agent dynamics. 
The control problems are constrained by a PDE of continuity-type without diffusion, governing the dynamics of the probability distribution of the agent population. 
We derive existence of optimal controls in a measure-theoretical setting as natural limits of finite agent optimal controls without any assumption on the regularity of control competitors. 
In particular, we prove the consistency of mean-field optimal controls with corresponding underlying finite agent ones.
The results follow from a $\Gamma$-convergence argument constructed over the mean-field limit, which stems from leveraging the superposition principle.}
\end{abstract}

\noindent {\bf Keywords}: finite agent optimal control, mean-field optimal control, $\Gamma$-convergence, superposition principle

\section{Introduction}
{\MF In the mathematical modelling of biological,  social, and economical  phenomena, self-organization of multi-agent interaction systems has become a focus
of applied mathematics and physics and mechanisms are studied towards the formation of global patterns. 
In the last years there has been a vigorous development of literature describing collective behaviour of interacting agents \cite{cucker2011general, cucker2008flocking, cucker2007emergent, gregoire2004onset, Jadbabaie2003correction, ke2002self, vicsek1995novel}, 
towards modeling phenomena in biology, such as motility and cell aggregation \cite{camazine2003self, keller1970initiation, koch1998social, perthame2007transport},  coordinated animal motion 
\cite{ballerini2008interaction, carrillo2009double, chuang2007state, couzin2003self, couzin2005effective, 
cristiani2010modelling, cucker2007emergent, niwa1994self, parrish1999complexity, parrish2002self, romey1996individual, toner1995long, yates2009inherent}, 
coordinated human \cite{cristiani2011multiscale, cucker2004foundations, short2008statistical} and synthetic agent interactions and behaviour, as in the case of 
cooperative robots \cite{chuang2007multi, leonard2001virtual, perea2009extension, sugawara1997cooperative}. 
Part of the literature is particularly focused on studying corresponding mean-field equations in order to simplify models for large populations of interacting agents: 
the effect of all the other individuals on any given individual is described by a single averaged effect. 
As it is very hard to be exhaustive in accounting all the developments of this very fast growing field, we refer to \cite{carrillo2014derivation, carrillo2017review, carrillo2010particle, choi2017emergent, vicsek2012collective} for recent surveys.
\\
Self-organization  is an incomplete concept, see, e.g. \cite{bongini2015unconditional}, as it is not always occurring when needed. In fact, local interactions between agents can be interpreted as distributed controls, which however are not always able to lead to global coordination or pattern formation. This motivated the research also of centralized optimal controls for multi-agent systems, modeling the intervention of an external government to induce desired dynamics or pattern formation. In this paper we are concerned with the control of deterministic multi-agent systems of the type
\begin{equation}\label{eq:state0}
  \dot x_i(t)=\fF^N(x_i(t),\xx(t))+ u_i(t),\quad i=1,\cdots, N.
\end{equation}
The map $\fF^N:\R^d\times (\R^d)^N\to \R^d$ models the interaction between the agents and 
$\uu$ represents the action of an external controller on the system. The control is optimized by minimization of a cost functional
\begin{equation}
  \label{finite_functional0}
    \cE^N(\xx,\uu):=\media_0^T \frac 1N\sum^N_{i=1} L^N(x_i(t),\xx(t))\,\d t+
  \media_0^T \frac 1N\sum^N_{i=1}\psi(u_i(t))\,\d t,
\end{equation}
where $L^N$ is a suitable cost function used for modelling the goal of the control and capturing the work done to achieve it, and $\psi$ is an appropriate positive convex function, which is superlinear at infinity and models the effective cost of employing the control. 
When the number $N$ of agents is very large, dynamical programming for solving the optimal control problem defined by minimization of \eqref{finite_functional0} under the constraints  \eqref{eq:state0} become computationally intractable. In fact, Richard Bellman coined the term ``curse of dimensionality'' precisely to describe this phenomenon. 

For situations where agents are indistinguishable, {\GGG e.g.}~drawn independently at random from an initial probability distribution $\mu_0$, and the dynamics $\fF^N(x_i(t),\xx(t))=\fF^N(x_i(t),\mu^N_t)$ depends in fact from the empirical distribution $\mu_t^N=\frac{1}{N} \sum_{i=1}^N \delta_{x_i(t)}$ one may hope to invoke again the use of mean-field approximations for a tractable (approximate) solution of the control problem. By formally considering the mean-field limit of the system \eqref{eq:state0} for $N\to \infty$ one obtains the continuity equation of Vlasov-type
\begin{equation}\label{eq:vlasov0}
  \partial_t \mu_t+\nabla\cdot\Big((\fF(x,\mu_t)+\vv_t)\mu_t\Big)=0
  \quad
  \text{in }(0,T)\times \R^d,
\end{equation}
where $\mu$ 
{\GGG is the weak limit of $\mu^N$ and represents}
the (time dependent) probability distribution of agents, and  
{\GGG $\nnu=\vv\mu$ is a suitable vector control  measure absolutely continuous w.r.t.~$\mu$ and
subjected to a cost functional 
 \begin{equation}
  \label{infinite_functional0}
    \cE(\mu,\nnu):=\media_0^T \int_{\R^d} L(x,\mu_t)\,\d\mu_t(x)\,\d t + 
    \GGG \media_0^T \int_{\R^d}\psi(\vv(t,x))\,\d\mu_t(x)\,\d t,
\end{equation}
The vector measure $\nnu=\vv\mu$ can in fact be obtained as the weak limit
of the sequence of finite dimensional control measures
\begin{equation}
  \label{eq:73}
  \nnu^N=\media_0^T \delta_t\otimes \nnu^N_t\,\d t,\quad
  \nnu^N_t = \frac{1}{N} \sum_{i=1}^N u_i(t) \delta_{x_i(t)},
  \quad t\in [0,T].
\end{equation}
{\GGG Under suitable assumptions on $\psi$, on the convergence of 
$\fF^N$ to $\fF$ and of $L^N$ to $L$, and assuming for simplicity that 
the initial data are confined in a compact subset of $\R^d$, 
one of the} main result of this paper can be summarized as follows.
\begin{theorem}\label{main0}
If the initial measures $\mu_0^N=\frac{1}{N} \sum_{i=1}^N \delta_{x^N_i(0)}$ 
weakly converge to a limit probability measure $\mu_0$ then 
the minimum $E^N(\mu_0^N)$ of \eqref{finite_functional0} among all the solution
of the controlled system \eqref{eq:state0}
converges to the minimum $E(\mu_0)$ of the functional \eqref{infinite_functional0}
among all the solutions of \eqref{eq:vlasov0} with initial datum $\mu_0$. 
Moreover, all the accumulation points (in the topology of weak convergence of measures) of the 
measures associated to minimizers $\xx^N,\uu^N$ of \eqref{finite_functional0}
are minima of \eqref{infinite_functional0}.
\end{theorem}
}
The idea of solving finite agent optimal control problems by considering a mean-field approximation has been considered since the 1960's 
\cite{fleming1977generalized, florentin1961optimal, kushner1962optimal} with the introduction of stochastic optimal control. The optimal  control of stochastic differential equations 
\begin{equation}\label{eq:state0stoch}
  d X^i_t=\fF(X^i_t,\mu^N_t)+ \vv(t,X^i_t) + \sigma d W^i_t,\quad i=1,\cdots, N;
\end{equation}
with non-degenerate diffusion and independent Brownian motions $W^i$, has been for a long while studied via the optimal control of the law $\mu_t = \operatorname{Law}(X_t)$ constrained by a McKean-Vlasov equation
\begin{equation}\label{eq:mckeanvlasov0}
  \partial_t \mu_t+\nabla\cdot\Big((\fF(x,\mu_t)+\vv(t,x) )\mu_t\Big)= \sigma \Delta \mu_t, 
\end{equation} 
under a suitable control cost
 \begin{equation}
  \label{infinite_functional1}
    \cE(\mu,\vv):=  \media_0^T \int_{\R^d} L(x,\mu_t)\,\d\mu_t(x)\,\d t + \media_0^T \int_{\R^d} \psi(\vv(t,x)) \d\mu_t(x)\,\d t.
\end{equation}
Most of the literature on  stochastic control  is focused primarily on the solution of McKean-Vlasov optimal control problems. 
The most popular methods are based on extending Pontryagin's maximum principle \cite{albi2017mean, andersson2011maximum, bensoussan2013mean, buckdahn2011general, carmona2013control}  
or deriving a dynamic programming principle, and with it a form of a Hamilton-Jacobi-Bellman equation on a space of probability measures \cite{Bayraktar2018randomized, lauriere2014dynamic, pham2018bellman}.
However, the rigorous justification that the McKean-Vlasov optimal control problem is consistent with the limit of optimal controls for stochastic finite agent models has been proved surprisingly just very recently \cite{lacker2017limit}. 
The techniques used in the latter paper are largely based on martingale problems, combining ideas from the McKean-Vlasov limit theory with a well-established compactification method for stochastic control \cite{fleming1977generalized}.\\
Due to the probabilistic nature of the methods used for stochastic control, the consistency results become weaker for the limiting case of vanishing viscosity $\sigma\to0$ as in \eqref{eq:vlasov0} (see e.g. \cite{lacker2017limit}). 
Sharp results for purely deterministic dynamics \eqref{eq:state0} require indeed measure-theoretical methods. 
The first work addressing the consistency of mean-field optimal control for deterministic finite agent systems is \cite{fornasier2014mean1}. 
In the latter paper an analogous result as Theorem \ref{main0} is derived for general penalty functions $\psi$ with polynomial growth, 
including the interesting case of linear growth at $0$ and  infinity, motivated by results of sparse controllability for finite-agent models 
\cite{bongini2014sparse, bongini2017sparse, caponigro2013sparse}. Other models of sparse mean-field optimal control have been considered in \cite{albi2016invisible, bongini2017mean, fornasier2014mean2}.
The generality of the  penalty function $\psi$ in \cite{fornasier2014mean1} has required  to {\GGG
restrict the class of controls: they have been assumed to be} locally Lipschitz continuous in space feedback control functions $u_i(t)=\vv(t,x_i)$ with controlled time-dependent Lipschitz constants. 

In this paper and in our main result Theorem \ref{main0} we remove this restriction, but we {\GGG
still impose suitable coercivity on the admissible controls, by}
requiring the function $\psi$ to have superlinear growth to infinity. As sparsity of controls, i.e. the localization of controls in space, is mainly due to the linear behaviour of the penalty function at $0$, the superlinear growth at infinity does not exclude the possibility of using this model for sparse control. Moreover, in this framework there is no need for enforcing a priori that controls are smooth feedback functions of the state variables and the limit process comes very natural in a measure-threoretical sense. In view of the minimal smoothness required to the governing interaction functions $\fF^N, \fF$ (they are assumed to be just continuous) there is no uniqueness of solutions in general of \eqref{eq:state0} and \eqref{eq:vlasov0}. Hence, the main results of mean-field limit are derived by leveraging the powerful machinery of the superposition principle \cite[Section 3.4]{Ambrosio2014continuity}.
\\

The paper is organized as follows: after recalling  in details the notation 
{\GGG and a few preliminary results on optimal transport, doubling functions and convex functionals on measures in Section 2, we describe our setting of optimal control problems in Section 3, 
together with the precise statements of our main results.} 
We address the existence of solutions of the finite agent optimal control problem in Section 4. Crucial moment estimates are derived in Section 5 for feasible competitors for the mean-field  control problem, which are useful for deriving compactness arguments further below. Section 6 is dedicated to {\GGG the proofs of 
our main Theorems. A relevant part} is devoted to Theorem \ref{main0} by developing a $\Gamma$-convergence argument. While the $\Gamma-\liminf$ inequality follows by relatively standard lower-semicontinuity arguments, the derivation of the  $\Gamma-\limsup$ inequality requires a technical application of the superposition principle. Equi-coercivity and convergence of minimizers follow from compactness arguments based on moment estimates from Section 5.
}

\section{Notation and preliminary results}

Throughout the paper we work with $\R^d$ as a state space and we fix a time horizon $T > 0$. 
\GGG We will denote by $\lambda$ the normalized restriction of the Lebesgue measure
to $[0,T]$, $\lambda:=\frac 1T \mathscr L^1\mres{[0,T]}$.\nc

Given $(\cS, d)$ a metric space, we use the classical notation $AC([0,T];\cS)$ for the classes of $\cS$-valued absolute continuous curves. 
We indicate with $\cM(\R^d), \cM(\R^d;\R^d)$ the space of Borel (vector-valued) measures.

\subsection{Probability measures and optimal transport costs}
\label{subsec:prob}

We call \GGG $\cP(\R^d)$ the space of Borel probability measures. 
If $f:\Omega\to \R^h$ is a Borel map defined in a Borel subset $\Omega$ of $\R^d$,
and $\mu\in \cP(\R^d)$ is concentrated on $\Omega$, we will denote by 
$f_\sharp\mu$ the Borel measure in $\R^h$ defined by $f_\sharp\mu(B):=\mu(f^{-1}(B))$,
for every Borel subset $B\subset \R^h$.

Whenever $\psi:\R^d\to [0,+\infty]$ is a lower semicontinuous function, we set
\begin{equation}
  \label{eq:36}
  \calC_\psi(\mu_0,\mu_1):=\inf 
  \left\lbrace \int_{\R^d \times \R^d} \psi(y-x) \,\d\gamma(x,y): \gamma \in \Pi(\mu_0,\mu_1) \right\rbrace ,
\end{equation}
where $\Pi(\mu_0,\mu_1)$ is the set of the optimal transport plans:
\begin{equation*}
\Pi(\mu_0,\mu_1):= \lbrace \gamma \in \cP(\R^d \times \R^d): \gamma (B \times \R^d)= \mu_0(B), \gamma (\R^d \times B)= \mu_1(B) \quad \forall \, B \text{ Borel set in } \R^d \rbrace. 
\end{equation*}
In the particular case when $\psi(z):=|z|$, $z\in \R^d$, \eqref{eq:36} defines the $L^1$-Wasserstein distance
\begin{equation}
\label{eq:37}
W_1(\mu_0, \mu_1) := \inf \left\lbrace \int_{\R^d \times \R^d} |x-y| \,\d\gamma(x,y): \gamma \in \Pi(\mu_0,\mu_1) \right\rbrace;
\end{equation}
the infimum in \eqref{eq:37} is always finite and attained if $\mu_0,\mu_1$ belong to the space
$\cP_1(\R^d)$ of Borel probability measure with finite first order moment:
\begin{equation}
  \label{eq:39}
  \cP_1(\R^d):=\Big\{\mu\in \cP(\R^d):\int_{\R^d}|x|\,\d\mu(x)<+\infty\Big\}.
\end{equation}
$\cP_1(\R^d)$ endowed with $W_1(\mu_0, \mu_1)$ is a complete and separable metric space.
\GGG In particular we will consider absolutely continuous curves $t\mapsto \mu_t$ in $AC([0,T];\cP_1(\R^d))$. 
They will canonically induce a parametrized measure $\tilde\mu:=\int \delta_t\otimes \mu_t\,\d \lambda(t) $
in $\cP_1([0,T]\times \R^d)$, satisfying
\begin{equation}
  \label{eq:28}
  \int f(t,x)\,\d\tilde\mu(t,x)=\int_0^T \int_{\R^d}f(t,x)\,\d\mu_t(x)\,\d\lambda(t)
  =\media_0^T \int_{\R^d}f(t,x)\,\d\mu_t(x)\,\d t.
\end{equation}
Convergence with respect to $W_1$ is equivalent to weak convergence (in duality with continuous and bounded functions) supplemented with convergence of first moment; equivalently, 
for every sequence $(\mu_n)_{n\in \N}\subset \cP_1(\R^d)$ and candidate limit $\mu\in \cP_1(\R^d)$ 
\begin{equation}
  \label{eq:41}
  \lim_{n\to\infty}W_1(\mu_n,\mu)=0
  \quad
  \Leftrightarrow\quad
  \lim_{n\to\infty}\int\zeta\,\d\mu_n=
  \int\zeta\,\d\mu
  \quad
  \text{for every }\zeta\in C(\R^d),\
  \sup_{x\in \R^d}\frac{\zeta(x)}{1+|x|}<\infty.
\end{equation}
\nc
In $\cP_1(\R^d)$ we consider the subset $\cP^N(\R^d)$ of discrete measures
\begin{displaymath}
  \cP^N(\R^d):=\left\{\mu=\frac 1N\sum_{i=1}^N \delta_{x_i} \text{ for some }x_i\in \R^d\right\}.
\end{displaymath}
A measure $\mu$ belongs to $\cP^N(\R^d)$ if and only if
$\#\supp(\mu)\le N$ and $N\mu(B)\in \N$ for every Borel set $B$ of $\R^d$.
Let us now fix an integer $N\in \N$ and consider vectors $\xx=(x_1,\cdots,x_N) \in (\R^d)^N$;
we will use the notation $\sigma: (\R^d)^N\to (\R^d)^N$ to 
denote a permutation of the coordinates of vectors in $(\R^d)^N$ and we set
\begin{displaymath}
  \sfd_N(\xx,\yy):=\min_\sigma \frac 1N\sum_{i=1}^N |x_i-\sigma(\yy)_{i}|,\quad
  |\xx|_N:=\sfd_N(\xx,\oo)=
  \frac 1N\sum_{i=1}^N |x_i|.
\end{displaymath}
To every vector $\xx\in (\R^d)^N$ we can associate the measure 
$\mu[\xx]:=\frac 1N \sum_{i=1}^N \delta_{x_i}\in \cP^N(\R^d)$ 
and we notice that by \eqref{eq:66}
\cite[Theorem 6.0.1]{Ambrosio2008gradient}
\begin{equation}
  \label{eq:71}
  \sfd_N(\xx,\yy)=W_1(\mu[\xx],\mu[\yy]),\quad
  |\xx|_N=\int_{\R^d} |x|\,\d\mu[\xx](x)=
  W_1(\mu[\xx],\delta_0).
\end{equation}

\noindent From now on we say that a map $\gG^N:\R^d\times (\R^d)^N\to \R^k$ is symmetric if 
\begin{displaymath}
  \gG^N(x,\yy)=\gG^N(x,\sigma(\yy))\quad\text{for every permutation  }\sigma:(\R^d)^N\to (\R^d)^N.
\end{displaymath}
Given a symmetric and continuous map $\gG^N$ we can associate
a function defined on measures $G^N:\R^d\times \cP^N(\R^d)\to \R^k$ by setting
\begin{equation}
  \label{eq:4}
  G^N(x,\mu[\yy]):=\gG^N(x,\yy).
\end{equation}
Throughout the paper we use the following notion of convergence for symmetric maps: 
\begin{definition}\label{d:convergence}
We say that a {\MF sequence} of symmetric maps $G^N$, $N\in \N$, $\cP_1$-converges to 
$G:\R^d\times \cP_1(\R^d)\to \R^k$ uniformly on compact sets as $N\to +\infty$ if
for every sequence of measure $\mu_k\in \cP^{N_k}(\R^d)$ converging to $\mu$ in
$\cP_1(\R^d)$ as $N_k\to\infty$ we have
\begin{equation}\label{eq:6}
  \lim_{k\to+\infty} \sup_{x \in C} \left| G^{N_k}(x,\mu_k)- G(x,\mu)\right|  = 0, \quad\text{for every compact }C \subset \R^d.
\end{equation}
\end{definition}

\subsection{Doubling and moderated convex functions}
\begin{definition}\label{def:Young}
We say that $\phi: [0,+\infty) \to [0,+\infty)$
is an \emph{admissible} 
function if
$\phi(0)=0$, $\phi$ is strictly convex \GGG and of class $C^1$ with $\phi'(0)=0$, \nc
superlinear at $+\infty$,
and doubling, i.e., there exists $K>0$ 
\GGG such that 
\begin{equation}
\phi(2r)\leq K\big(1+\phi(r)\big)\quad
\text{for any $r\in[0,+\infty).$}
\label{eq:19}
\end{equation}
\GGG Let $U$ be a subspace of $\R^d$.
We say that a convex function $\psi:U\to [0,+\infty)$ 
is \emph{moderated} if there exists an admissible function $\phi:[0,+\infty)  \to [0,+\infty)$ and a constant $C>0$
such that 
\begin{equation}\label{hp:phi}
\phi(|x|)-1 \leq \psi(x) \leq C(1 + \phi(|x|)) \quad \text{for every }x \in U.
\end{equation} 
\end{definition}
{\MF By convexity, an admissible function $\phi$ satisfies
  $\phi(r) + \phi'(r)(s-r) \leq \phi(s)$ \GGG
for every $r,s\in [0,+\infty)$; 
in particular choosing $s=0$ and \MF $s=2r$ one obtains
\begin{equation}\label{eq:doubl2}
\GGG 0\le \phi(r)\le \MF r \phi'(r) \leq (\phi(2r) - \phi(r)) \leq 
\GGG  K\big(1+\phi(r)\big)\quad
\text{for every }r\in [0,+\infty).
\end{equation}
}%
\GGG 
It is not difficult to see that if a differentiable convex function $\phi$ satisfies
\begin{equation}
  \label{eq:21}
  r\phi'(r)\le A(1+\phi(r))\quad\text{for every }r\ge R,
\end{equation}
for some constants $A,R>0$, then $\phi$ satisfies \eqref{eq:19} with 
$K=\max(\mathrm e^A,\max_{[0,2R]}\phi)$. In fact, differentiating the function
$z\mapsto (\phi(z r)+1)$ for $z\in [1,D]$ and $r\ge R$
we get
$\frac{\partial }{\partial\theta}\big(\phi(\theta r)+1\big)=r\phi'(\theta r)\le 
A(1+\phi(\theta r))$
so that 
\begin{equation}
\phi(Dr)\le (\phi(r)+1)\mathrm e^{(D-1)A}
\quad D>1,\ r>R.\label{eq:70}
\end{equation}
In particular \eqref{eq:19} yields
\begin{equation}
\phi(Dr)\le (\phi(r)+1)\mathrm e^{(D-1)K}
\quad D>1,\ r>0.\label{eq:70bis}
\end{equation}
We also recall that $\phi'$ is monotone, i.e.
  \begin{equation}
    (\phi'(r)- \phi'(s))(r-s) \ge 0\quad
    \text{for every }s,r\ge0.\label{eq:62}
\end{equation}

The next lemma shows that it is always possible to 
approximate a convex superlinear function
by a monotonically increasing sequence of moderated ones.

\begin{lemma}
  \label{le:approximation}
  Let $\UU$ be a subspace of $\R^d$ and $\psi:U \to [0,+\infty]$ be a superlinear function
  with $\psi(0)=0$. 
  \begin{enumerate}
  \item There exists an admissible function $\theta:[0,+\infty)\to[0,+\infty)$
    such that 
    \begin{equation}
      \label{eq:38}
     \psi(x)\ge \theta(|x|)-\frac 12\quad\text{for every $x\in \UU$.} 
    \end{equation}
  \item If $\psi$ is also convex, then there exists
  a sequence $\psi^N:\UU\to [0,+\infty)$, $N\in \N$, of \emph{moderated convex functions} such
  that 
  \begin{equation}
    \label{eq:20}
    \psi^N(x)\le \psi^{N+1}(x),\qquad
    \psi^N(x)\uparrow \psi(x)\quad\text{as }N\to +\infty\quad\text{for every }x\in \UU.
  \end{equation}
  \end{enumerate}
\end{lemma}

\begin{proof}
  It is not restrictive to assume $\UU=\R^d$.\\
  \textbf{Claim 1.}   
  Let us set $h(r):=\min_{|x|\ge r}\psi(x)$ and $\bar n:=\min\big\{n\ge 0:h(2^n)\ge 1\big\}$,
  $\bar r:=2^{\bar n}$.
  The map $h:[0,+\infty)\to [0,+\infty]$
  is increasing, lower semicontinuous, and satisfies $\lim_{r\to\infty}h(r)/r=+
  \infty$.
  By a standard result of convex analysis (see e.g.~\cite[Lemma 3.7]{rossi2003tightness})
  there exists a convex superlinear function $k:[0,+\infty)\to [0,+\infty)$ such that 
  $h(r)\ge k(r)$ for every $r\in [0,+\infty)$ so that 
  $\psi(x)\ge k(|x|)$ for every $x\in \R^d$. 
  
  Let us define the sequence $(a_n)_{n\in \N}$ by induction:
  \begin{equation}
    \label{eq:23}
    \begin{aligned}
      a_n&:=0\quad\text{for every $n\in \N$, $n< \bar n$};\quad a_{\bar
        n}:=2^{-\bar n},\\
     a_{n+1}&:=\min \Big(2a_n,
      2^{-(n-1)}\big(k(2^n)-k(2^{n-1})\big)\Big) \quad\text{for every
      }n\ge \bar n.
    \end{aligned}
  \end{equation}
  Since $k$ is
  convex and increasing, the sequence $n\mapsto a_n$ is positive and increasing;
  since $k$ is superlinear, it is also easy to check that $\lim_{n\to\infty}a_n=+\infty$.
  
  We now consider the piecewise linear continuous function $\theta_1:[0,+\infty)\to[0,+\infty)$ 
  on the dyadic partition $\{0,2^0,2^1,2^2,\cdots, 2^n,\cdots\}$, $n\in \N$,
  satisfying
  \begin{equation}
    \label{eq:22}
    \theta_1(r)\equiv 0\quad \text{if }0\le r\le \bar r=2^{\bar n},\quad
    \theta_1'(r)=a_n\quad\text{if $2^n<r<2^{n+1}$},\quad
    n\in\N,\ n\ge \bar n.
  \end{equation}
  Since $\theta_1'\le k'$ a.e.~in $[0,+\infty)$ we have $\theta_1\le k$. 
  Moreover, by construction, for
  $0\le r\le 2\bar r$ we have 
  $\theta_1(r)\le \theta_1(2\bar r)=1$ and $\theta_1'(2r)\le 2 \theta_1'(r)$ if $r\ge\bar r$ so that 
  $\theta_1$ is also doubling since
  $$\theta_1(2 r)=\theta_1(2\bar r)+\int_{\bar r}^r 2\theta_1'(2s)\,\d s\le 
  1+4\int_{\bar r}^r \theta_1'(s)\,\d s=1+4\theta_1(r)\quad\text{for every }r\ge \bar r.$$ 
  Replacing now $\theta_1$ by the convex combination
  $\theta_2(r):=\frac 12\theta_1(r)+\frac 12 r/\bar r$ we get a strictly increasing function, still satisfying 
  \eqref{eq:38}.

  By possibly replacing $\theta_2$ with $\theta_3(r):=\int_{r-1}^r \theta(s)\,\d s$ (where we set $\theta_2(s)\equiv \theta_2(0)=0$
  whenever $s<0$) we obtain a $C^1$ function. 
  Strict convexity can be eventually obtained by taking the convex combination
  $\theta(r):=(1-\eps)\theta_3(r)+\eps(\sqrt{1+r^2}-1)$ for a sufficiently small $\eps>0$.

  \textbf{Claim 2.} 
  Notice that the function $x\mapsto \theta_2(|x|)$ is convex. 
  We can define $\psi^N$ by inf-convolution:
  \begin{equation}
    \label{eq:24}
    \psi^N(x):=\inf_{y\in \R^d} \psi(y)+N\theta_2(|x-y|),\quad
    x\in \R^d.
  \end{equation}
  It is easy to check that the infimum in \eqref{eq:24} is attained, $\psi^N$ is convex
  (since it is the inf-convolution of two convex functions) and satisfies the obvious bounds
  \begin{equation}
    \label{eq:25}
    \psi^N(x)\le N\theta_2(|x|),\quad \psi^N(x)\le \psi(x),\quad
    \psi^N(x)\le \psi^{N+1}(x)\quad
    \text{for every }x\in \R^d.
  \end{equation}
  In particular $\psi^N$ is continuous; since $x\mapsto \theta_2(x)$ is continuous at $x=0$
  and $\theta_2(|x|)\ge \frac 1{2\bar r}|x|$ we easily get 
  $\lim_{N\to\infty}\psi^N(x)=\psi(x)$ for every $x\in \R^d.$

  It remains to show that $\psi^N$ is moderated. 
  Since $\psi(x)\ge \theta_2(|x|)-1/2$ 
  and for every $y\in \R^d$ the triangle inequality yields $\min (|x-y|,|y|)\ge |x|/2$,
  we get
  \begin{equation}
    \label{eq:26}
    \psi^N(x)+1/2\ge \inf_{y\in \R^d} \theta_2(|y|)+N\theta_2(|x-y|)\ge \theta_2(|x|/2)\ge 
    \frac 14\theta_2(|x|)-\frac 14
  \end{equation}
  and the bounds
  \begin{equation}
    \label{eq:27}
    \frac 14\theta(|x|)-\frac 34\le \psi^N(x)\le 4N\frac14\theta_2(|x|).
  \end{equation}
  By possibly replacing $\theta_2$ with $\theta$ we conclude.
  \end{proof}
Let us make explicit two simple applications of the properties of Definition \ref{def:Young}.

\begin{remark}
  \label{rem:moment-reinforcement}
  \upshape
  If $\cK\subset \cP_1(\R^d)$ is a relatively compact set and $\psi:\UU\to[0,+\infty]$ is 
  a superlinear funcion defined in a subspace $U$ of $\R^d$ with $\psi(0)=0$,
  then there exists an admissible 
  function $\theta:[0,+\infty)\to [0,+\infty)$ such that 
  \begin{equation}
    \label{eq:42}
    \sup_{\mu\in \cK}\int_{\R^d}\theta(|x|)\,\d\mu(x)<\infty,\quad 
    \theta(|x|)\le 1+\psi(x)\quad\text{for every }x\in \UU.
  \end{equation}
  In fact, Prokhorov theorem yields the tightness 
  of the set $\tilde \cK:=\{|x|\mu:\mu\in \cK\}$ of finite measures, so that we can find a superlinear function $\alpha:\R^d\to[0,\infty)$ 
  such that 
   \begin{equation}
    \label{eq:42bis}
    \sup_{\mu\in \cK}\int_{\R^d}\alpha(x)\,\d\mu(x)<\infty.
  \end{equation}
  We can then apply the first statement of Lemma \ref{le:approximation} 
  with superlinear function $\alpha\land \psi$.
\end{remark}

\begin{lemma}
  \label{le:convergence}
  Let $\zeta:\R^d\to[0,+\infty)$ be a moderated convex function with $\zeta(0)=0$ and let $\mu^i_n\in\cP_1(\R^d)$, $i=0,1$, 
  be two sequences
  converging to $\mu$ in $\cP_1(\R^d)$ and let $\gamma_n$ be the optimal plan attaining the minimum in 
  \eqref{eq:37} for $W_1(\mu^0_n,\mu^1_n)$. If
  \begin{equation}
    \label{eq:43}
    \limsup_{n\to\infty}\int\zeta\,\d\mu^i_n\le 
    \int\zeta\,\d\mu
  \end{equation}
  then 
  \begin{equation}
    \label{eq:44}
    \lim_{n\to\infty}\int\zeta(y-x)\,\d\gamma_n(x,y)=0,\quad
    \lim_{n\to\infty}\calC_\zeta(\mu_n^0,\mu_n^1)=0.
  \end{equation}
\end{lemma}
\begin{proof}
  Let $\phi$ be an admissible function satisfying \eqref{hp:phi}
  for $\psi:=\zeta$.
  We observe that for every $x,y\in \R^d$
  \begin{equation}
    \label{eq:45}
    \phi(|y-x|)\le \phi(|x|+|y|)
    \le K\Big(1+\phi(\tfrac 12|x|+\tfrac12|y|)\Big)
    \le K\Big(1+\phi(|x|)+\phi(|y|)\Big).
  \end{equation}
Inequality  \eqref{eq:43} shows that $\zeta$ is uniformly integrable w.r.t.~$\mu_n$ 
  (see \cite[Lemma 5.1.7]{Ambrosio2008gradient}) so that 
  \begin{equation}
    \label{eq:46}
    \lim_{n\to\infty}\int\phi(|x|)\,\d\mu^i_n(x)=
    \int\phi(|x|)\,\d\mu(x), \qquad i=1,2,
  \end{equation}
 whence 
  \begin{equation}
    \label{eq:47}
    \lim_{n\to\infty}\int\Big(\phi(|x|)+\phi(|y|)\Big)\,\d\gamma_n(x,y)=
    2\int\phi(|x|)\,\d\mu(x)=
    \int\Big(\phi(|x|)+\phi(|y|)\Big)\,\d\gamma(x,y)
  \end{equation}
  where $    \gamma:=(x,x)_\sharp \mu$ is the weak limit of $\gamma_n$.
  It follows that the function $(x,y)\mapsto \phi(|x|)+\phi(|y|)$ is 
  uniformly integrable with respect to $\gamma_n$ 
  so that, by \eqref{eq:45} and 
  \cite[Lemma 5.1.7]{Ambrosio2008gradient}
  \begin{equation}
    \label{eq:48}
    \lim_{n\to\infty}\int \phi(|y-x|)\,\d\gamma_n(x,y)=
    \int \phi(|y-x|)\,\d\gamma(x,y)=0.
  \end{equation}
  Since $\zeta(y-x)\le C(1+\phi(|y-x|))$ by \eqref{hp:phi} we get \eqref{eq:44}.
\end{proof}
\nc

\subsection{Convex functionals on measures}
\label{appendix:convex_funct_meas}

We are concerned with the main properties of functionals defined on measures, for a detailed treatment of this subject we refer to \cite{ambrosio2000functions}.
Let $\psi: \R^h \to [0,+\infty]$ be a proper, l.s.c., convex and superlinear function, so that  
its  recession function $\sup_{r >0} \frac{\psi(rx)}{r}=\infty$
for all $x \neq 0$; we will also assume $\psi(0)=0$.

Let now $\Omega$ be an open subset of some Euclidean space
$\mu \in \cM^+(\Omega)$ be a reference measure and
$\nnu \in \cM(\Omega;\R^h)$ a vector measure; 
\nc we define the following functional
\begin{equation}
\label{eq:74}
\Psi(\nnu |\mu) := \int_{\Omega} \psi(\vv(x))\,\d\mu(x) \quad
\text{if }\nnu = \vv \mu\ll\mu,\quad
\Psi(\nnu |\mu):=+\infty\quad\text{if }\nnu\not\ll\mu.
\end{equation}
We state the main lower semicontinuity result 
for the functional $\Psi$.
\begin{theorem}\label{t.funct:meas:lsc}
Suppose that we have two sequences $\mu_n \in \cM^+(\Omega)$, 
$\nnu_n \in \cM(\Omega;\R^h)$ weakly  converging to 
$\mu \in \cM^+(\Omega)$ and $\nnu \in \cM(\Omega,\R^h)$, respectively. 
Then 
\[ \liminf_{n \to+ \infty} \Psi(\nnu_n | \mu_n) \geq \Psi(\nnu |\mu). \]
\GGG
In particular, if $\liminf_{n \to+ \infty} \Psi(\nnu_n | \mu_n) <+\infty$, 
we have $\nnu\ll\mu$.
\end{theorem}
The proof can be found in \cite{Ambrosio2008gradient}, Lemma 9.4.3.

\section{The optimal control problem and main results}
\label{intro:problem}

\paragraph{Cost functional.}

{\MF Assume that we are given}
a sequence of functions $L^N:\R^d\times (\R^d)^N\to [0,+\infty)$, $N \in \N$, and a function  $L:\R^d\times \cP_1(\R^d) \to [0,+\infty)$ such that $L^N$ is continuous and symmetric
for every $N\in\N$ and $L$ is continuous.  
We assume that 
\begin{equation}\label{hp:Lconv}
\mbox{$L^N$ $\cP_1$-converges to $L$ uniformly on compact sets, as $N\to \infty$, }
\end{equation}
in the sense of Definition \ref{d:convergence}.
\\
{\MF Assume that we are given} 
\GGG 
\begin{equation}
  \begin{gathered}
    \text{a subspace $U\subset \R^d$ and a moderated convex function
      $\psi: U \to [0,+\infty)$ with $\psi(0)=0$}.
  \end{gathered}
\label{eq:40}
\end{equation}
We will also fix an auxiliary function $\phi$ satisfying \eqref{hp:phi}.

\MF Typical examples we consider for $\psi$  include
\begin{itemize} 
\item $\psi(x)=\frac 1p|x|^p, \quad p>1$;
\item 
\GGG
$\psi(x)= \frac1p|x|$ for $|x| \leq 1$ and $\psi(x) = \frac1p|x|^p $ for $|x| > 1$, $ p>1.$
\end{itemize}
\nc 
Denoting by $U^N$ the Cartesian product, we define a cost functional 
$\cE^N: AC([0,T];(\R^d)^N) \times L^1([0,T]; \UU ^N) \to [0,+\infty)$
by
\begin{equation}
  \label{finite_functional}
    \cE^N(\xx,\uu):=\media_0^T \frac 1N\sum^N_{i=1} L^N(x_i(t),\xx(t))\,\d t+
  \media_0^T \frac 1N\sum^N_{i=1}
  \psi (u_i(t))\,\d t.
\end{equation}
{\MF We consider also another cost functional 
 $\cE: AC([0,T];\cP_1(\R^d)) \times \cM([0,T] \times \R^d; \UU ) \to [0,+\infty)$
  defined by}
\GGG (recall \eqref{eq:28})
 \begin{equation}
  \label{infinite_functional}
  \cE(\mu,\nnu):=\media_0^T \int_{\R^d} L(x,\mu_t)\,\d\mu_t(x)\,\d t + 
  \GGG
  \Psi(\nnu|\tilde\mu),
\end{equation}
\GGG
where $\Psi$ is defined as in \eqref{eq:74}.
Notice that if $\Psi(\nnu|\tilde\mu)<\infty$ then 
$\nnu=\vv\mu$ for a Borel vector field $\vv\in L^1_{\tilde\mu}([0,T]\times \R^d;\UU )$ 
so that for $\lambda$-a.e.~$t\in [0,T]$ the measure
$\nnu_t:=\vv(t,\cdot)\mu_t$ belongs to $\cM(\R^d;\UU )$ and we can write 
\begin{equation}\label{eq:defPsi}
  \Psi(\nnu |\tilde\mu) =
  \int_{[0,T]\times \R^d}\psi(\vv(t,x))\,\d\tilde\mu(t,x)=
  \media_0^T \int_{\R^d}\psi(\vv(t,x))\,\d\mu_t\,\d t=
  \media_0^T \Psi(\nnu_t|\mu_t)\,\d t.
\end{equation}
\nc
{\MF We shall prove below that the functional $\cE$ is 
the $\Gamma$-limit of $\cE^N$ in suitable sense \cite{dalmaso1993introduction}.}

\paragraph{The constraints (State equations).}

{\MF Assume that we are given} a sequence of functions $\fF^N: \R^d \times (\R^d)^N \to \R^d$, $N\in\N$, symmetric and continuous
and a continuous function $\fF:\R^d\times \cP_1(\R^d)\to \R^d$.
We assume that there exist constants $A,B\geq 0$
such that
\begin{equation}\label{hp:linear}
 |\fF^N(x,\yy)|\le A+B(|x|+|\yy|_N),\qquad\qquad |\fF(x,\mu)|\le A+B\Big(|x|+\int_{\Rd} |y|\,\d\mu(y)\Big),
 \end{equation}
\GGG and $\fF^N$, $\fF$ and $\UU$ satisfy the \emph{compatibility} condition
\begin{equation}
  \label{eq:65}
  \fF^N(x,\yy)-\fF(x,\mu)\in \UU\quad
  \text{for every }x\in \R^d,\ \yy\in (\R^d)^N,\ \mu\in \cP_1(\R^d).
\end{equation}
Moreover, we assume that 
\begin{equation}\label{hp:Fconv}
\mbox{$\fF^N$ $\cP_1$-converges to $\fF$ uniformly on compact sets, as $N \to \infty$,}
\end{equation}
in the sense of Definition \ref{d:convergence}.

Given $\uu=(u_1, \ldots, u_N)\in L^1([0,T];\UU^N)$, a control map, 
we consider the system of differential equations
\begin{equation}\label{eq:state}
  \dot x_i(t)=\fF^N(x_i(t),\xx(t))+ u_i(t),\quad i=1,\cdots, N.
\end{equation}
{\MF The map $\fF^N:\R^d\times (\R^d)^N\to \R^d$ models the interaction between the agents  and 
$\uu$ represents the action of an external controller on the system. }
For every $\uu\in L^1([0,T];\UU^N)$ and $\xx_0\in(\R^d)^N$, thanks to \eqref{hp:linear} and the continuity of $\fF^N$,
there exists a global solution, in the Carath\'eodory sense,  $\xx\in AC([0,T];(\R^d)^N)$ of \eqref{eq:state} such that $ \xx(0)=\xx_0$. 
Since we have assumed only the continuity of the velocity field $\fF^N$, uniqueness of solutions is not guaranteed in general.
We then define the non empty set 
\[
	\AA^N:=\{(\xx,\uu)\in AC([0,T];(\R^d)^N\times L^1([0,T];\UU^N):  \mbox{ $\xx$ and $\uu$  satisfy \eqref{eq:state}, $\cE^N(\xx,\uu)<\infty$} \}.
\]
Moreover we also define for every $\xx_0\in (\R^d)^N$ the non empty set
\[
	\AA^N(\xx_0):=\{(\xx,\uu)\in\AA^N : \xx(0)=\xx_0\}.
\]
{\GGG Every initial vector $\xx_0=(x_{0,1},\cdots,x_{0,N})\in (\R^d)^N$ gives raise to 
  the empirical distribution
  \begin{equation}
    \label{eq:33}
    \mu_0=\mu[\xx_0]:=\frac 1N\sum_{i=1}^N \delta_{x_{0,i}}.
  \end{equation}
  Similarly, every curve $\xx\in AC([0,T];(\R^d)^N)$
  is associated to the curve of probability measures
  \begin{equation}
    \label{eq:32}
   \mu=\mu[\xx]\in AC([0,T];\cP_1(\R^d)):
   \qquad\mu_t:=\frac 1N\sum_{i=1}^N\delta_{x_i(t)},\quad t\in [0,T],
  \end{equation}
  and every pair $(\xx,\uu)\in AC([0,T];(\R^d)^N)\times  L^1([0,T];\UU^N)$
  is linked to the \emph{control} vector measure 
\begin{equation}
  \label{eq:30}
  \nnu=\nnu[\xx,\uu]\in \cM([0,T]\times\R^d;\UU ):\quad
  \nnu:=\int_0^T \delta_t\otimes \nnu_t\,\d\lambda,\quad
   \nnu_t:=\frac{1}{N} \sum_{i=1}^N u_i(t)\delta_{x_i(t)}.
\end{equation}
We will show that 
for every choice of solutions and controls $(\xx^N,\uu^N)\in \AA^N(\xx^N_0)$ 
such that 
the cost functional $\cE^N(\xx^N,\uu^N)$ remains uniformly bounded
and the initial empirical distributions $\mu_0^N=\mu[\xx_0^N]$ 
is converging to a limit measure $\mu_0$ in $\cP_1(\R^d)$
a mean-field approximation holds: 

\begin{theorem}[Compactness]
  \label{thm:compactness}
  Let $(\xx^N_0)_{N\in \N}$ be a sequence of initial data in
  $(\R^d)^N$ such that the empirical measure
  $\mu_0^N=\mu[\xx^N_0]$ converges to a
  probability measure $\mu_0 $ in $\cP_1(\R^d)$ as
  $N\to \infty$, and let $(\xx^N,\uu^N)\in \AA^N(\xx^N_0)$ 
such that 
the cost functional $\cE^N(\xx^N,\uu^N)$ remains uniformly bounded.
Up to extraction of a suitable subsequence, the
  empirical measures $\mu^N=\mu[\xx^N]$
  converge uniformly in $\cP_1(\R^d)$ to a curve of probability
  measures $\mu\in AC([0,T];\cP_1(\R^d))$, the control measures
  $\nnu^N=\nnu[\xx^N,\uu^N]$ converge to a limit control measure
  $\nnu$ weakly$^*$ in $\cM([0,T] \times \R^d; \UU )$, and $(\mu,\nnu)$ fulfills the
  continuity equation
  \begin{equation}\label{eq:vlasov}
    \partial_t \mu_t+\nabla\cdot\Big(\fF(x,\mu_t)\mu_t+\nnu_t\Big)=0
    \quad\text{in }(0,T)\times\R^d
  \end{equation}
  in the sense of distributions.
\end{theorem}
\nc
Motivated by the above result, we define the non empty set
\begin{align*}
	\AA:=\Big\{&(\mu,\nnu)\in AC([0,T];\cP_1(\R^d))\times \cM([0,T] \times \R^d; \UU ) :\\
  &\mbox{ $\mu$ and $\nnu$  satisfy \eqref{eq:vlasov} in the sense of distributions},\ 
    \cE(\mu,\nnu)<\infty\Big\},
\end{align*}
and its corresponding subset associated to a given initial measure 
$\mu_0\in\cP_1(\R^d)$:
\[
	\AA(\mu_0):=\{(\mu,\nnu)\in \AA : \mu(0)=\mu_0\}.
\]

{\MF The elements of $\AA^N$ can be interpreted as the trajectories $(x_1, \ldots, x_N)$ of $N$ agents along with their strategies $(u_1, \ldots, u_N)$,
whose dynamics is described by the system of ODEs \eqref{eq:state}. Analogously, the elements of $\AA$ can be interpreted as the trajectories of a continuous or discrete distribution of agents
whose dynamics is described by the PDE \eqref{eq:vlasov} 
{\GGG under} the action of an external controller described by the measure $\nnu$.}
\paragraph{The minimum problems.}

The objective of the controller is to minimize the cost functional $\cE^N$ (resp. $\cE$). 
{\MF We consider the following optimum sets, defined by corresponding optimal control problems}:
\begin{align}
  E^N(\xx_0):=&
  \min_{ (\sxx,\suu) \in\AA^N(\sxx_0)}\cE^N(\xx,\uu),
                      &
  P^N(\xx_0):=  {}&\argmin\{\cE^N(\xx,\uu):  (\xx,\uu) \in\AA^N(\xx_0) \}, \label{minN0}\\
E(\mu_0):=&
  \min_{ (\mu,\snnu) \in\AA(\mu_0)}\cE(\mu,\nnu),&
	 P(\mu_0):= {}& \argmin\{\cE(\mu,\nnu):  (\mu,\nnu) \in\AA(\mu_0) \},\label{min0}
\end{align}
\GGG 
where we suppose that $\mu_0\in D(E):=\{\mu\in \cP_1(\R^d):\AA(\mu)\text{ is not empty}\}.$ \nc

We are interested in the rigorous justification of the convergence of the control problem \eqref{minN0}
towards the corresponding infinite dimensional one \eqref{min0}.

\paragraph{Main results.}

{\MF We state now more formally our main result} concerning
the sequence of functionals $\cE^N$ to $\cE$, inspired to $\Gamma$-convergence.
\begin{theorem}[$\Gamma$-convergence]
  \label{th:main}
	The following properties hold:
	\begin{itemize}
\item \textbf{$\Gamma-\liminf$ inequality}: for every $(\mu,\nnu)\in AC([0,T];\cP_1(\R^d)) \times \cM([0,T] \times \R^d; \UU )$ 
and every  sequence $(\xx^N,\uu^N)\in AC([0,T];(\R^d)^N) \times L^1([0,T]; \UU^N)$ 
such that $\mu[\xx^N] \to \mu \text{ in } C([0,T];\cP_1(\R^d))$, $\nnu[\xx^N,\uu^N] 
\rightharpoonup^* \nnu$ in $\cM([0,T] \times \R^d; \UU )$,  we have
\begin{equation}\label{gammaliminf}
\liminf_{N \to \infty} \cE^N(\xx^N,\uu^N)\geq\cE(\mu,\nnu).
\end{equation}
\item \textbf{$\Gamma-\limsup$ inequality}: 
for every $(\mu,\nnu)\in \AA$ 
such that 
\begin{equation}
\GGG
\int_{\R^d}\phi(|x|)\,\d\mu_0(x)<\infty\label{eq:61}
\end{equation}
there exists a sequence  $(\xx^N, \uu^N)\in\AA^N$  
\GGG with $x_{0,i}^N\in \supp(\mu_0)$ for every $i=1,\cdots, N$, \nc
such that 
\begin{gather}
\label{eq:64}
\mu[\xx^N]\to \mu \text{ in } C([0,T];\cP_1(\R^d)),\quad
\nnu[\xx^N,\uu^N]\rightharpoonup^* \nnu 
\text{ in }\cM([0,T] \times \R^d; \UU ),\\
\GGG
\lim_{N\to\infty}\frac1N\sum_{i=1}^N\phi(|x_{0,i}^N|)=
\int_{\R^d}\phi(|x|)\,\d\mu_0(x),
\label{eq:63}
\end{gather}
and
\begin{equation}\label{limsup}
\limsup_{N \to \infty} \cE^N(\xx^N,\uu^N) \leq \cE(\mu,\nnu).
\end{equation}
\end{itemize}
\end{theorem}
\GGG As a combination of Theorem \ref{thm:compactness} 
and Theorem \ref{th:main} we obtain the convergence of minima.
\begin{theorem}\label{t:minima}
Let $\mu_0\in \cP_1(\R^d)$ be
satisfying \eqref{eq:61}.
\begin{enumerate}
\item There exists a sequence $\xx_0^N\in (\R^d)^N$, $N\in \N$, satisfying
  \begin{gather}
    \label{eq:49}
    \lim_{N\to\infty}W_1(\mu[\xx_0^N],\mu_0)=0,\\
    \label{eq:69}
    \limsup_{N\to\infty}\frac 1N\sum_{i=1}^N\phi(|x_{0,i}^N|)=\int\phi(|x|)\,\d\mu_0(x),\\
    \lim_{N\to\infty}E^N(\xx_0^N)=E(\mu_0)
    \label{eq:68}.
  \end{gather}
\item If a sequence $\xx_0^N$ satisfies \eqref{eq:49}
  then 
  for every choice of $(\xx^N,\uu^N)\in P(\xx_0^N)$ with
  $\mu^N:=\mu[\xx^N]$ and $\nnu^N:=\nnu[\xx^N,\uu^N]$, the collection
  of limit points $(\mu,\nnu)$ of $(\mu^N,\nnu^N)$ in
  $C([0,T];\cP_1(\R^d))\times \cM([0,T] \times \R^d; \UU )$ is non
  empty and contained in $P(\mu_0)$.
\item If moreover $U=\R^d$ and $\mu_0$ has compact support, 
  then 
  every sequence $(\xx_0^N)_{N\in \N}$ satisfying \eqref{eq:49} 
  and uniformly supported in a compact set
  also satisfies \eqref{eq:69} and \eqref{eq:68}.
\end{enumerate}
\end{theorem}
\nc

\subsection{Examples}

\paragraph{First order examples.}
  Take a continuous function $H:\R^d\to \R^d$ satysfying
  \[
  |H(x)|\leq A+B|x| \qquad \forall x\in\R^d
  \]
   and set
  \begin{equation}
    \label{eq:9}
    \fF^N(x,\yy):=\frac 1N\sum_{j=1}^N H(x-y_j)=
    \int_{\R^d} H(x-y)\,\d\mu[\yy](y)
  \end{equation}
  and
   \begin{equation}
    \label{eq:9bis}
    \fF(x,\mu):= \int_{\R^d} H(x-y)\,\d\mu(y).
  \end{equation}
  When $H=-\nabla W$ for an even function $W\in C^1(\R^d)$ the system \eqref{eq:state} is associated to the gradient flow of the interaction energy $\cW:(\R^d)^N\to\R$ defined by
  \begin{equation}
    \label{eq:10}
    \GGG \cW(\xx):=\frac 1{2N^2}\sum_{i,j=1}^NW(x_i-x_j)
  \end{equation}
  with respect to the weighted norm $\|\xx\|^2=\frac 1N\sum_{i=1}^N|x_i|^2$.

  More generally, we can consider a continuous kernel $K(x,y):\R^d\times \R^d\to \R^d$ 
  satisfying 
  \[
  |K(x,y)|\leq A+B(|x|+|y|) \qquad \forall x,y\in\R^d,
  \]
  obtaining
  \begin{equation}
    \label{eq:14}
      \fF^N(x,\yy):=\frac 1N\sum_{j=1}^N K(x,y_j)=
      \int_{\R^d} K(x,y)\,\d\mu[\yy](y)
  \end{equation}
and
   \begin{equation}
    \label{eq:14bis}
    \fF(x,\mu):= \int_{\R^d} K(x,y)\,\d\mu(y).
  \end{equation}

An example for $L^N$ and $L$ is the variance:
\[
L^N(x,\xx):=\left|x-\frac 1N\sum_{j=1}^N x_j\right|^2,
\]
and
\[
L(x,\mu):=\left|x-\int_{\R^d} y\,\d\mu(y)\right|^2.
\]

\paragraph{A second order example.}
\GGG
Second order systems can be easily reduced to first order models, if we admit
controls on positions and velocities. 
Let us see an example where controls act only on the velocities. \nc
  Assume $d=2m$ and write the vector $x=(q,p)$, where $q\in \R^m$ denotes the position and $p\in \R^m$ the velocity.
  
  \GGG We consider the  vector field $\fF^N(x,\xx)=(\fF^N_1(x),\fF^N_2(x,\xx))$ defined by
  \begin{equation}
    \label{eq:11}
    \fF^N_1((q,p))=p,\qquad
    \fF^N_2((q,p),(\qq,\pp))=-\frac 1N \sum_{j=1}^N \nabla W(p-p_j),
  \end{equation}
where the first component $\fF_1^N$ is local and it is not influenced by the 
  interaction with the other particles. \nc
  
 We are interested to the system 
  \begin{equation}
    \label{eq:12}
    \begin{cases}
    \dot q_i=p_i,\\
    \dot p_i=-\frac 1N \sum_{j=1}^N \nabla W(p_i-p_j)+u_{i},
    \end{cases}
  \end{equation}
  which corresponds to \eqref{eq:state} 
  where the vector $\uu$ has the particular form \GGG
  $\uu=((0,u_1),\cdots,(0,u_N))$, 
  \nc so that it is constrained
  to the subspace $\UU^N$ \GGG where $\UU=\{(0,u):u\in \R^m\}\subset \R^{2m}$. \nc
  The limit vector field $\fF(x,\mu)=(\fF_1(x),\fF_2(x,\mu))$ is defined by
  \begin{equation}
    \label{eq:11bis}
    \fF_1((q,p))=p,\qquad
    \fF_2((q,p),\mu)=- \nabla_p W*\mu,
  \end{equation}
  and the continuity equation
  \begin{equation}
    \partial_t \mu_t+\nabla\cdot\Big(\fF(x,\mu_t)\mu_t+\nnu_t\Big)=0.
  \end{equation}
   becomes a Vlasov-like equation
  \begin{equation*}
  \partial_t \mu_t+p\cdot\nabla_q\mu_t+\nabla_p\cdot\big(\fF_2(x,\mu_t)\mu_t+\nnu_t\big)=0.
  \end{equation*}
  \GGG It is easy to check that this structure fits in our abstract setting, 
  since $\fF^N,\fF$ satisfy the compatibility condition \eqref{eq:65}:
  for every $x\in \R^d$, $\yy\in (\R^d)^N$ and
  $\mu\in \cP_1(\R^d)$ we have
  $\fF(x,\mu)-\fF^N(x,\yy)=\big(0,\fF_2(x,\mu)-\fF_2^N(x,\yy)\big)\in \UU^N$.
  \nc

  By choosing in \eqref{eq:11}
  \begin{equation}
    \label{eq:13}
    \fF^N_2((q,p),(\qq,\pp))=-\alpha p-\frac 1N \sum_{j=1}^N \nabla W(p-p_j)
  \end{equation}
  for some $\alpha>0$ we obtain a model with friction in the velocity part.
By choosing in \eqref{eq:11}
  \begin{equation}
    \label{eq:13bis}
    \fF^N_2((q,p),(\qq,\pp))=-\frac 1N \sum_{j=1}^N a(|q-q_j|)(p-p_j)
    \end{equation}
    where $a:[0,+\infty)\to \R_+$ is a continuous and nonincreasing (thus bounded)
    function, we obtain a model of alignment.
A particular and interesting example for $a$ is given by the following decreasing function
 $a(|q|)=1/(1+|q|^2)^\gamma$ for some $\gamma\geq 0$, which yields the Cucker-Smale 
flocking model \cite{cucker2007emergent, cucker2007mathematics}.

An example for $L^N$ and $L$ in the second order model is the variance of the velocities:
\[
L^N((q,p),(\qq,\pp)):=\left|p-\frac 1N\sum_{j=1}^N p_j\right|^2,
\]
and
\[
L((q,p),\mu):=\left|p-\int_{\R^d} r_2\,\d\mu(r_1,r_2)\right|^2.
\]

\section{The finite dimensional problem}

Here we discuss the well-posedness of the finite dimensional control problem \eqref{minN0}. 

A first estimate on the solution is presented in the following Lemma, where we use the notation 
$|\yy|_N = \frac{1}{N}\sum_{i=1}^N |y_i|$, with $\yy = (y_1,\ldots, y_N)\in(\R^d)^N$.

\begin{lemma}\label{l.boundedness_x}
Let $(\xx,\uu)\in\AA^N$. Then 
\begin{equation}\label{momest}
\sup_{t \in [0,T]} |\xx(t)|_N \leq \left( |\xx(0)|_N + AT + \int_0^T |\uu(s)|_N\,\d s\right)e^{2BT},
\end{equation}	
where $A$ and $B$ are the constants of the assumption \eqref{hp:linear}.
\end{lemma}

\begin{proof}
From the integral formulation of equation \eqref{eq:state} we get 
\begin{equation}
\begin{split}
|x_i(t)| &\leq |x_i(0)| + \int_0^t |\fF^N(x_i(s),\xx(s))|\d s + \int_0^t |u_i(s)|\,\d s \\
&\leq |x_i(0)| + \int_0^t \left( A+B(|x_i(s)|+|\xx(s)|_N) \right)\d s + \int_0^t |u_i(s)|\,\d s. \\
\end{split}
\end{equation}
Averaging with respect to $N$ we obtain
\begin{equation}
|\xx(t)|_N \leq |\xx(0)|_N + AT+ \int_0^T |\uu(s)|_N\,\d s + 2B\int_0^t |\xx(s)|_N \d s
\end{equation}
and we conclude by Gronwall lemma.
\end{proof}

\begin{proposition}\label{p:sol_control}
\GGG For every $N\in\N$ 
and $\xx_0\in(\R^d)^N$ \nc the minimum problem 
\eqref{minN0} admits a solution, i.e., the set $P^N(\xx_0)$ is not empty.
\end{proposition}
\begin{proof}
We fix  $N\in\N$ and $\xx_0\in(\R^d)^N$.
Let $\lambda := \inf \{\cE^N(\xx,\uu): (\xx,\uu)\in\AA^N(\xx_0)\}$.
Since  $\AA^N(\xx_0)$ is not empty, $\lambda<+\infty$.
Let $(\xx^k,\uu^k) \in \AA^N(\xx_0)$ be a minimizing sequence and $C:=\sup_k \cE^N(\xx^k,\uu^k) <+\infty$.

Since 
\begin{equation}\label{eq.bound_uk}
\sup_k \media_0^T \psi(u_i^k(t))\,\d t \leq C, \qquad \forall \, i=1,\ldots, N,
\end{equation}
and the function $\psi$ is superlinear, then the sequence $\uu^k$ is equi-integrable 
and hence weakly relatively compact in $L^1([0,T];\UU^N)$.
Hence there exists $\uu \in L^1([0,T], \UU^N)$ and a subsequence, {\MF again} denoted by $\uu^k$,
weakly convergent to $\uu$ in $L^1([0,T], \UU^N)$.  

Thanks to Lemma \ref{l.boundedness_x} the associated trajectories $\xx^k$ are equi-bounded. 
Let us now show the equi-continuity of $x^k_i(t)$. For $s \leq t$, by the equation \eqref{eq:state} we have
\begin{equation}\label{eq:eqk}
x_i^k(t) - x_i^k(s) = \int_s^t \fF^N(x_i^k(r), \xx^k(r))\,\d r + \int_s^t u_i^k(r)\,\d r. 
\end{equation}
Using the growth condition \eqref{hp:linear} and \eqref{momest} we get
\begin{equation*}
\begin{split}
|\xx^k(t) - \xx^k(s)|_N &\leq \frac1N\sum_{i=1}^N\int_s^t |\fF^N(x_i^k(r), \xx^k(r))|\,\d r + \int_s^t |\uu^k(r)|_N\,\d r \\
&\leq A(t-s) + 2B\int_s^t |\xx^k(r)|_N \,\d r + \int_s^t |\uu^k(r)|_N\,\d r \\
&\leq A(t-s) + 2B\left( |\xx_0|_N + AT + \int_0^T |\uu^k(r)|_N\,\d r \right)e^{2BT}(t-s)  + \int_s^t |\uu^k(r)|_N\,\d r .\\
\end{split}
\end{equation*}
Since $\int_0^T |\uu^k(r)|_N\,\d r $ is bounded, we have
\begin{equation}\label{equicont}
\sup_k |\xx^k(t) - \xx^k(s)|_N \leq \tilde C|t-s| + \sup_k\left|\int_s^t |\uu^k(r)|_N\,\d r\right|, \qquad \forall\, s,t \in [0,T],
\end{equation}
where $\tilde C:= A + 2B\left( |\xx_0|_N + AT + \sup_k \int_0^T |\uu^k(r)|_N\,\d r \right)e^{2BT}$. 
By the equi-integrability of $\uu^k$, the inequality \eqref{equicont} shows the  
 equi-continuity of $\xx^k$.
By Ascoli-Arzel\`a theorem there exists a continuous curve $\xx$ and a subsequence, {\MF again} denoted by $\xx^k$
such that $\xx^k\to\xx$ in $C([0,T];(\R^d)^N)$.
Passing to the limit in \eqref{eq:eqk} we obtain 
\begin{equation}
x_i(t) - x_i(s) = \int_s^t \fF^N(x_i(r),\xx(r))\,\d r + \int_s^t u_i(r) \,\d r, \qquad  i=1,\ldots, N,
\end{equation}  
from which we deduce that $\xx$ is absolutely continuous and solves the equation \eqref{eq:state}. 
Hence $(\xx,\uu)\in \AA^N(\xx_0)$.

Finally, by the convexity of $\psi$ and the continuity of $L^N$ we obtain the  
lower semicontinuity property
\begin{equation}
\begin{split}
\liminf_{k} \cE^N(\xx^k,\uu^k) &= \liminf_k \left[\media_0^T \frac 1N\sum_{i=1}^N  L^N(x_i^k(t), \xx^k(t))\,\d t+
\frac 1N\sum_{i=1}^N\media_0^T \psi(u^k_i(t))\,\d t \right] \\
&\geq  \media_0^T \frac 1N\sum_{i=1}^N L^N(x_i(t),\xx(t))\,\d t+
\frac 1N\sum_{i=1}^N\media_0^T \psi(u_i(t))\,\d t,
\end{split}
\end{equation}
whence the minimality of $(\xx,\uu)\in \AA^N(\xx_0)$.
\end{proof}

\section{Momentum estimates}

In this section we study the set $\AA$.
We observe that if $(\mu,\nnu)\in\AA$, then 
for any $\zeta \in C_c^1(\R^d)$ we have that the map $t \mapsto \int_{\R^d} \zeta \d\mu_t$ is absolutely continuous, {\MF a.e. differentiable}, and
\begin{equation}
\label{eq:54}
\frac{d}{dt}\int_{\R^d}\zeta(x)\,\d\mu_t(x) = \int_{\R^d} \la \ff(t,x),\nabla\zeta(x) \ra \,\d\mu_t(x) 
+ \int_{\R^d} \la\nabla\zeta(x), \d\nnu_t(x) \ra \quad \mbox{for a.e. } t\in[0,T],
\end{equation}
for the vector field $  \ff(t,x):=\fF(x,\mu_t)$ satisfying the structural bounds
\begin{equation}
  \label{eq:53}
  |\ff(t,x)|\le A+B\Big(|x|+\int_{\R^d}|x|\,\d\mu_t\Big).
\end{equation}

\GGG
In order to highlight the structural assumptions needed for the apriori estimates of this section, 
we introduce the set 
\begin{equation}
  \begin{aligned}
    \tilde\AA:=\Big\{&(\mu,\nnu,\ff):
    \mu\in AC([0,T];\cP_1(\R^d)),\ 
    \nnu\in \cM([0,T]\times
    \R^d;\UU ),\ 
    \cE(\nnu,\tilde\mu)<\infty,\\
    &\ff:[0,T]\times\R^d\to\R^d \text{ Borel function satisfying 
    \eqref{eq:54} and
    \eqref{eq:53}}\Big\}\label{eq:55};
  \end{aligned}
\end{equation}
the above discussion shows that if $(\mu,\nnu)\in \AA$ then setting $\ff(t,x):=\fF(x,\mu_t)$ 
we have $(\mu,\nnu,\ff)\in \tilde\AA$.\nc

Firstly, let us show a uniform bound in time of the first moment, which 
is the infinite dimensional version of Lemma \ref{l.boundedness_x}.
\begin{lemma}\label{p.first_moment}
\GGG
If $(\mu,\nnu,\ff)\in\tilde\AA$ \nc
then the following estimate holds true
\begin{equation}
\label{eq:59}
\sup_{t \in [0,T]} \int_{\R^d}|x|\,\d\mu_t(x) \leq  \left( \int_{\R^d}|x|\,\d\mu_0(x) + AT + |\nnu|((0,T)\times\R^d) \right)e^{2BT}.
\end{equation}
\GGG
In particular, there exists a constant $M>0$ only depending on $A,B, T, \cE(\mu,\nnu)$ and
$\int_{\R^d}|x|\,\d\mu_0$ such that
\begin{equation}
  \label{eq:56}
  |\ff(t,x)|\le M(1+|x|)\quad\text{for every }(t,x)\in [0,T]\times\R^d.
\end{equation}
\end{lemma}
\begin{proof}
Let $\zeta \in C_c^1(\R^d)$ be a cut-off function such that $0 \leq \zeta \leq 1$,  
\begin{equation*}
\zeta(x) = 
\begin{system}
1 \qquad \text{ if } |x| \leq 1,\\
0 \qquad \text{ if } |x| \geq 2,
\end{system}
\end{equation*}
and $|\nabla\zeta|\leq 1$.
Let $\zeta_n$ be the sequence $\zeta_n(x):= \zeta(x/n)$.
Consider now the product $\zeta_n(x)|x|$ and smooth it out in zero by substituting $|x|$ with $g_\varepsilon(x):= \sqrt{|x|^2 + \varepsilon}$. 
Now $\zeta_ng_\varepsilon$ is a proper test function and the following equality holds true
\begin{equation*}
\begin{split}
\int_{\R^d} &\zeta_n(x)g_\varepsilon(x) \d\mu_t(x) -\int_{\R^d} \zeta_n(x)g_\varepsilon(x)\d\mu_0(x) \\
&= \int_0^t \int_{\R^d} \la \ff(s,x), \nabla (\zeta_n(x)g_\varepsilon(x))\ra  \d\mu_s(x)\d s 
+  \int_0^t \int_{\R^d} \la \nabla (\zeta_n(x)g_\varepsilon(x)), \d\nnu_s(x)\ra \d s 
\end{split}
\end{equation*}
Thanks to 
\[ |\nabla\zeta_n(x)| \leq \frac{1}{n}, \quad  g_\varepsilon(x) \leq |x| + \sqrt{\varepsilon}, 
\quad |\nabla g_{\varepsilon}(x)| = \frac{|x|}{\sqrt{|x|^2 +\varepsilon}} \leq 1\] 
we can write
\begin{equation*}
\begin{split}
\int_{\R^d} &\zeta_n(x)g_\varepsilon(x) d\mu_t(x) - \int_{\R^d} \zeta_n(x)g_\varepsilon(x)d\mu_0(x) \\
&  \leq  \left( 1 + \frac{\sqrt{\varepsilon}}{n} \right)\int_0^t \int_{\R^d} |\ff(s,x)|\d\mu_s(x)\d s 
+  \left( 1 + \frac{\sqrt{\varepsilon}}{n} \right)\int_0^t \int_{\R^d} d|\nnu_s|(x)\d s.
\end{split}
\end{equation*}
Apply now monotone convergence as $\varepsilon \to 0$ first, then  let $n\to \infty$. Owing to $\zeta_n|x| \nearrow |x|$ we get 
\begin{equation}
\begin{split}
\int_{\R^d} &|x|\d\mu_t(x) - \int_{\R^d} |x|\d\mu_0(x) \leq  
\int_0^t \int_{\R^d} | \ff(s,x)| \d\mu_s(x) \d s + |\nnu|((0,T)\times\R^d)\\
&\leq \int_0^t \int_{\R^d} \left[ A+B\Big(|x|+\int_{\R^d} |x|\,\d\mu_s(x)\Big)\right] \d\mu_s(x)\d s +  |\nnu|((0,T)\times\R^d) \\
&\leq  AT + 2B\int_0^t \int_{\R^d}|x|\d\mu_s(x)ds +  |\nnu|((0,T)\times\R^d),
\end{split}	
\end{equation}
and we conclude by Gronwall inequality.
\end{proof}

\begin{lemma}\label{p:distr:lip}
If $(\mu,\nnu,\ff)\in\tilde\AA$ with $\nnu=\vv\tilde\mu$,
then for any $\vartheta \in C^1_{\Lip}(\R^d)$ the following equality holds
\begin{equation}
\frac{d}{dt}\int_{\R^d}\vartheta(x)\,\d\mu_t(x) = \int_{\R^d} \la \ff(t,x) + \vv(t,x),\nabla\vartheta(x) \ra \,\d\mu_t(x)  	
\qquad \mbox{for a.e. } t\in[0,T],
\end{equation}
where $C^1_{\Lip}(\R^d)$ denotes the space of continuously differentiable functions with bounded gradient.
\end{lemma}

\begin{proof}
Let $\vartheta \in C^1_{\Lip}(\R^d)$ and $\zeta_n$ the sequence of cut-off functions defined in the proof of Lemma \ref{p.first_moment}.
Then $\zeta_n \vartheta$ is a test function and
\begin{equation}
\begin{split}
\frac{d}{dt}&\int_{\R^d}\zeta_n(x)\vartheta(x)\d\mu_t(x) = \int_{\R^d} \la \ff(t,x) + \vv(t,x), \nabla(\zeta_n(x)\vartheta(x)) \ra \d\mu_t(x)\\
&=  \int_{\R^d} \la \ff(t,x) + \vv(t,x), \nabla\zeta_n(x) \vartheta(x) + \zeta_n(x) \nabla \vartheta(x) \ra \d\mu_t(x).\\
\end{split}
\end{equation}
Taking into account that $|\nabla \zeta_n| \leq \frac{1}{n}\chi_{B_{2n}}$, the Lipschitz {\MF continuity} of $\vartheta$,
the growth condition on $\ff$ and Lemma \ref{p.first_moment}, by dominated convergence we obtain that
\begin{equation}
\int_{\R^d}\vartheta(x) \d\mu_t(x) = \int_{\R^d}\vartheta(x) \d\mu_0(x) 
+ \int_{\R^d} \la \ff(t,x) + \vv(t,x), \nabla \vartheta(x) \ra \d\mu_t(x).
\end{equation}
\end{proof}

Now we are ready to prove the main result of this section. 
\GGG It involves an auxiliary admissible function $\theta:[0,\infty)\to[0,\infty)$ 
(according to Definition \ref{def:Young}) dominated by $\psi$, i.e.
\begin{equation}
  \label{eq:50}
  \theta(|x|)\le 1+\psi(x)\quad
  \text{for every }x\in U;
\end{equation}
notice that, combining Lemma \ref{le:approximation} and Remark \ref{rem:moment-reinforcement},
if $\mu_0\in \cP_1(\R^d)$ we can always find 
an admissible function $\theta$ satisfying \eqref{eq:50} and
\begin{equation}
  \label{eq:51}
  \int_{\R^d}\theta(|x|)\,\d\mu_0(x)<\infty.
\end{equation}
\begin{proposition}\label{t.phi_moment}
\GGG 
Let $(\mu,\nnu,\ff)\in\tilde\AA$ 
and let $\theta$ be
an admissible function satisfying \eqref{eq:50}
and \eqref{eq:51}.
\nc
Then there exists a constant 
$C >0$, depending only on $A$, $B$, $T$,  
$\int_{\R^d}|x|\,\d\mu_0(x)$, $\cE(\mu,\nnu)$,  
\GGG $\theta(1)$ and the doubling constant $K$ of $\theta$ (see \eqref{eq:19}), \nc
such that
\begin{equation}
\label{eq:58}
\sup_{t \in [0,T]} \int_{\R^d} \theta(|x|)\,\d\mu_t(x) {\MF \leq}C\left(1 + \int_{\R^d} \theta(|x|)\,\d\mu_0(x)\right).
\end{equation}
\end{proposition}

\begin{proof}
Since $\cE(\mu,\nnu) <+\infty$, we have that $\nnu=\vv\tilde\mu$. We also set $\vartheta(x):=\theta(|x|)$, $x\in \R^d$.

STEP 1: We start by approximating $\theta$ from below with a sequence of $C^1$-Lipschitz functions
\begin{equation}
\GGG \vartheta^n(x):=\theta^n(|x|),\quad
\theta^n(r) := 
\begin{cases}
\theta(r)&\text{if }|x| \leq n \\
\theta'(n)(r - n) + \theta(n)&\text{if }r > n.
\end{cases}
\end{equation}
\begin{equation}
\begin{split}
\int_{\R^d} \vartheta^n(x)\d\mu_t(x) &= \int_{\R^d} \vartheta^n(x)\d\mu_0(x) 
+ \int_0^t \int_{\R^d} \la \ff(s,x) + \vv(s,x),\nabla\vartheta^n(x) \ra \d\mu_s(x) \d s\\
&\leq \int_{\R^d} \vartheta^n(x)\d\mu_0(x) + 
\int_0^t \int_{\R^d} |\ff(s,x) + \vv(s,x)| |\nabla\vartheta^n(x)| \d\mu_s(x) \d s.
\end{split}
\end{equation}
\nc By construction, $\theta^n(|x|) \nearrow \theta(|x|)$
$|\nabla\vartheta^n(x)| \nearrow |\nabla \vartheta(x)|$, for every $x \in \R^d$;
we can thus pass to the limit in the relation above to get
\begin{equation}\label{eq:psi_ineq}
\int_{\R^d} \vartheta(x)\d\mu_t(x) \leq \int_{\R^d} \vartheta(x)\d\mu_0(x) 
+ \int_0^t \int_{\R^d}  |\ff(s,x) + \vv(s,x)| |\nabla\vartheta(x)| \d\mu_s(x) \d s.
\end{equation}

STEP 2: We want to estimate the right hand side of \eqref{eq:psi_ineq}. 
\GGG Since $\theta'(r)\ge 0$ by \eqref{eq:doubl2}
and $\left|\nabla \vartheta(x)\right| = \left| \theta'(|x|)\frac{x}{|x|} \right| = \theta'(|x|)$, so that
\GGG \eqref{eq:56} yields \nc
\begin{equation}
\GGG
\int_{\R^d} 
|\ff(s,x)| |\nabla\vartheta(x)|\d\mu_s(x) 
\leq M\int_{\R^d} (1+|x|)\theta'(|x|)\d\mu_s(x) 
\end{equation}
By the monotonicity of $\theta'$ \eqref{eq:62} in $[0,1]$ and \eqref{eq:doubl2}, we have
\begin{displaymath}
  (1+r)\theta'(r)\le 2K(1+\theta(1)+\theta(r))
\end{displaymath}
\nc so that 
\begin{equation}
\GGG\int_{\R^d} |\ff(s,x)| |\nabla\vartheta(x)|\d\mu_s(x) 
%
\leq  2MK\left(1 +\theta(1)+ \int_{\R^d}\theta(|x|) \d\mu_s(x)\right).
\end{equation}
Concerning the second term on the right hand side of \eqref{eq:psi_ineq}
\begin{equation}
\begin{split}
|\vv(s,x)| |\nabla\vartheta(x)| &\leq \theta(|\vv(s,x)|) + \theta^*(|\nabla \vartheta(x)|) \\
&= \theta(|\vv(s,x)|) + \theta^*(\theta'(|x|)) \\
&= \theta(|\vv(s,x)|) + \theta'(|x|)|x| - \theta(|x|) \\
&\GGG
\leq \theta(|\vv(s,x)|) + K(1+\theta(|x|)) ,
\end{split}
\end{equation}
where the equality $\theta(|x|) + \theta^*(\theta'(|x|)) = \theta'(|x|)|x|$ comes from the definition of the Fenchel conjugate $\theta^*$.
What we end up with is the following
\begin{equation}
\begin{split}
\int_0^t\int_{\R^d} &|\vv(s,x)| |\nabla\vartheta(x)|\d\mu_s(x) \d s \\
&\leq \int_0^T\int_{\R^d}\theta(|\vv(s,x)|)\d\mu_s(x) \d s +
\GGG Kt+ K\nc\int_0^t\int_{\R^d}\theta(|x|)\d\mu_s(x) \d s \\
&\leq T\cE(\mu,\nnu) + \GGG (1+K)T+ K\nc \int_0^t \int_{\R^d}\theta(|x|) \d\mu_s(x)\d s.
\end{split}
\end{equation}
Summing up the two estimates we obtain for every $t\in [0,T]$ and a suitable constant $C>0$
\begin{equation}
\int_{\R^d} \theta(|x|)\d\mu_t(x) \leq \int_{\R^d} \theta(|x|)\d\mu_0(x) + CT +C\int_0^t \int_{\R^d}\theta(|x|) \d\mu_s(x) \d s,
\end{equation}
and thanks to the Gronwall inequality we get
\begin{equation}
\int_{\R^d} \theta(|x|)\d\mu_t(x) \leq e^{CT}\left(  \int_{\R^d} \theta(|x|)\d\mu_0(x) + CT \right).
\end{equation}
\end{proof}
\nc

\section{Proof of the main Theorems.}
\subsection{The superposition principle}
\label{app:super}
We first recall the superposition principle for solutions of the continuity equation
\begin{equation}\label{eq:continuity}
\partial_t	\mu_t + \nabla \cdot(\ww(t,\cdot) \mu_t) = 0.
\end{equation}
Let us denote with $\Gamma_T$ the complete and separable
metric space of continuous functions from $[0,T]$ to $\R^d$ endowed with the sup-distance 
and introduce the evaluation maps $e_t: \Gamma_T \to \R^d$ defined by
$e_t(\gamma) := \gamma(t)$, for $t \in [0,T]$.
The following result holds
\begin{theorem}[Superposition principle]\label{t:prob_repr}
Let $\mu_t$ be a narrowly continuous {\MF weak} solution to \eqref{eq:continuity} with a velocity field $\ww$ satisfying
\begin{equation}
\int_0^T \int_{\R^d} |\ww(t,x)|\,\d \mu_t(x)\d t < +\infty.
\end{equation}
Then there exists $\pi \in \cP(\Gamma_T)$ concentrated on the set of curves $\gamma \in AC([0,T];\R^d)$ such that 
\begin{equation}
\dot{\gamma}(t) = \ww(t,\gamma(t)) \quad \text{for a.e. } t\in [0,T].
\end{equation} 
Moreover, $\mu_t = (e_t)_\#{\pi}$ for any $t \in [0,T]$, 
i.e.
\begin{equation}
\int_{\R^d} \vartheta(y) \,\d \mu_t(y) = \int_{\Gamma_T} \vartheta(\gamma(t))\,\d\pi(\gamma), \qquad \forall \, \vartheta \in C_b(\R^d).
\end{equation} 
\end{theorem}
For the proof we refer to \cite[3.4]{Ambrosio2014continuity}.

\subsection{$\Gamma$-convergence.}

Let us start with a preliminary lemma.
\begin{lemma}
Let $(\xx,\uu)\in AC([0,T];(\R^d)^N) \times L^1([0,T]; \UU^N)$ and 
\GGG $\mu=\mu[\xx],$ $\nnu=\nnu[\xx,\uu]$.
Then we have 
\begin{equation}\label{ineqpar}
	 \frac 1N\sum^N_{i=1}\psi(u_i(t)) \geq \Psi(\nnu_t | \mu_t) \qquad \mbox{for a.e.~} t\in [0,T].
\end{equation}
Moreover, if $(\xx,\uu)\in\AA^N$ then
\begin{equation}\label{eqpar}
	 \frac 1N\sum^N_{i=1}\psi(u_i(t)) =  \Psi(\nnu_t | \mu_t) \qquad \mbox{for a.e. } t\in [0,T].
\end{equation}
\end{lemma}
\begin{proof}
  \GGG 
  Let us first compute the density of $\nnu$ w.r.t.~$\tilde\mu$.
  We introduce the finite set $I_N:=\{1,2,\cdots,N\}$
  with the discrete topology and the normalized counting measure 
  $\sigma_N=\frac 1N\sum_{i=1}^N\delta_i$. 
  We can identify $\xx$ with a continuous map from $[0,T]\times I_N$ to $\R^d$, 
  $\xx(t,i):=x_i(t)$, so that $\mu_t=\xx(t,\cdot)_\sharp \sigma_N$.
  Similarly, we set $\uu(t,i):=u_i(t)$, where 
  $\uu:[0,T]\to \UU^N$ is a Borel representative.
  In order to represent $\tilde\mu$ and $\nnu$ it is useful to deal with
  the map $\yy:[0,T]\times I_N\to [0,T]\times \R^d$, $\yy(t,i):=(t,\xx(t,i))$, which yields
  $\tilde \mu=\yy_\sharp \big(\lambda\otimes\sigma_N\big)$ and
  $\nnu=\yy_\sharp \big(\uu\cdot (\lambda\otimes\sigma_N)\big).$
  We denote by $Y\subset [0,T]\times \R^d$ the range of $\yy$, 
  and by 
  $$X(t)=\{x\in \R^d:(t,x)\in Y\}=\{x\in \R^d: x_i(t)=x\text{ for some }i\in I_N\}$$ 
  its fibers.
  For every $(t,x)\in [0,T]\times \R^d$ we will also consider the set 
  \begin{equation}
    \label{eq:67}
    J(t,x):=\{i\in I_N:x_i(t)=x\}
    \text{ with its characteristic function}
    \quad 
    \nchi_{t,x}(i):=
    \begin{cases}
      1&\text{if }x_i(t)=x\\
      0&\text{otherwise.}
    \end{cases}
  \end{equation}
    For every $t\in [0,T]$ the collection $\{\nchi_{t,x}:x\in X(t)\}$ provides a partition of unity of $I_N$
    and for every $i\in I_N$ the map
    $(t,x)\mapsto \nchi_{t,x}$ is upper semicontinuous in $[0,T]\times \R^d$. 
  The conditional measures $\tilde\mu_{t,x}\in \cP(I_N)$ 
  are then defined by 
  \begin{equation}
\tilde\mu_{t,x}(J):=\sigma_N(J\cap J(t,x))/\sigma_N(J(t,x)),\quad
(t,x)\in Y;\label{eq:76}
\end{equation}
since for every $J\subset I_N$ 
  \begin{displaymath}
    \sigma_N(J\cap J(t,x))=\int_J \nchi_{t,x}\,\d\sigma_N
    =\frac 1N\sum_{i\in J}\nchi_{t,x}(i),
  \end{displaymath}
  the map
  $(t,x)\mapsto \sigma_N(J\cap J(t,x))$ is also upper semicontinuous 
  and $\tilde \mu_{t,x}$ is a Borel family.
  
  One immediately checks that $\tilde\mu_{t,x}$ provides a disintegration 
  (see e.g.~\cite[Theorem 5.3.1]{Ambrosio2008gradient})
  of $\lambda\otimes\sigma_N$ w.r.t.~the map $\yy$, i.e.
  \begin{displaymath}
    \lambda\otimes\sigma_N=\int \tilde\mu_{t,x}\,\d\tilde\mu(t,x).
  \end{displaymath}
  Since $\nnu=\yy_\sharp \big(\uu\cdot (\lambda\otimes\sigma_N)\big),$
  we eventually end up with the representation formula 
  for the Borel vector field $\vv$  
\begin{equation}
  \left\{
  %
        %
  \begin{aligned}
    \vv(t,x) :={}& \int_{I_N}\uu(t,i)\,\d\tilde\mu_{t,x}(i)=
    \frac{1}{\sharp J(t,x)} \sum_{i \in J(t,x)}u_i(t)
    &&\text{if }(t,x)\in Y,\\
        \vv(t,x)  :={}&0&&\text{otherwise.}
  \end{aligned}
  \right.
\end{equation}
In particular
\begin{equation}
	\nnu_t=\vv(t,\cdot)\mu_t,\quad
        \mu_t=\sum_{x\in X(t)}\frac{\sharp J(t,x)}N\delta_x
\end{equation} 
and consequently
\begin{equation}
  \label{eq:78}
  \Psi(\nnu_t| \mu_t)
  = \int_{\R^d}\psi(\vv(t,x))\,\d\mu_t(x)=
  \sum_{x\in X(t)}
    \frac {\sharp J(t,x)}N
    \psi\Big(\frac{1}{\sharp J(t,x)} \sum_{i \in J(t,x)}u_i(t)\Big).
\end{equation}
The convexity of $\psi$ immediately yields 
\begin{equation}\label{eq:ineqcontrol}
	\Psi(\nnu_t | \mu_t) 
        \leq \frac{1}{N}  \sum_{i=1}^N \psi(u_i(t)).
\end{equation}
Let us show that equality holds in \eqref{eq:ineqcontrol} if $(\xx,\uu)\in \AA^N$.

Let $\PP$ be the collection of all the partitions $P$ of $I_N$. 
It is clear that for every $t\in [0,T]$ the family $P_\sxx(t):=\{J(t,x): x\in X(t)\}$ is an element of $\PP$;
moreover for every $P\in \PP$ the set 
\begin{equation}
  \label{eq:77}
  S_P:=\{t\in [0,T]: P_\sxx(t)=P\}\quad\text{is Borel}.
\end{equation}
To show \eqref{eq:77} we introduce an order relation on $\PP$:
we say that $P_1\prec P_2 $ if every element of $P_1$ is contained in some 
element of $P_2$. We denote by $\hat P:=\{Q\in \PP:P\prec Q\}$ the collection of all the 
partitions $Q$ coarser than $P$. 

It is easy to check that for every $P\in \PP$ the set 
$P^{-1}_\sxx\big(\hat P\big)=\{t\in [0,T]:P_\sxx(t)\in \hat P\}$ is closed.
 In fact, if $P_\sxx(t)\not\in \hat P$ 
then there is a set $I\in P$ not contained in any element of $P_\sxx(t)$, so that we can find
two indices $i,j\in I$ belonging to different elements of $P_\sxx(t)$, i.e.~$x_i(t)\neq x_j(t)$.
By continuity, this relation holds in a neighborhood $U$ of $t$, so that $P_\sxx(s)\not\in \hat P$ 
for every $s\in U$.

Since for every partition $P\in \PP$ 
$\{P\}=\hat P\setminus\cup\big\{\hat Q:Q\in \hat P,Q\neq P\big\}$, it follows that 
\begin{displaymath}
  S_P=P_\sxx^{-1}\big(\hat P \big)\setminus 
  \bigcup_{Q\in \hat P,Q\neq P}P_\sxx^{-1}\big(\hat Q \big),
\end{displaymath}
so that $S_P$ is the difference between closed sets and \eqref{eq:77} holds.

We can therefore decompose the interval $[0,T]$ 
in the finite Borel partition $\{S_P:P\in \PP\}$.
On the other hand, for every partition $P\in \PP$ and every pair of indices $i,j$ in $I\in P$ 
we have $x_i(t)=x_j(t)$ in $S_P$ so that
$\dot x_i(t)=\dot x_j(t)$ for $\lambda$-almost every $t\in S_P$
and consequently, by \eqref{eq:state},
we obtain that $u_i(t)=u_j(t)$ for $\lambda$-a.e.~$t\in S_P$. We eventually deduce
\begin{equation}
  \label{eq:52}
  \sharp I\psi\Big(\frac{1}{\sharp I} \sum_{i \in I}u_i(t)\Big)=
  \sum_{i\in I}\psi(u_i(t))\quad\text{for every }I\in P_\sxx(t),\quad
  \text{$\lambda$-a.e.~in }S_P,
\end{equation}
and therefore, by \eqref{eq:78},
\newcommand{\ssxx}{{\boldsymbol{x}}}
\begin{align*}
  \Psi(\nnu_t | \mu_t)
  &=
    \sum_{I\in P_\ssxx(t)}
    \frac {\sharp I}N
    \psi\Big(\frac{1}{\sharp I} \sum_{i \in I}u_i(t)\Big)=
    \frac1N\sum_{I\in P_\ssxx(t)}\sum_{i\in I}\psi(u_i(t))
    \\&=
    \sum_{i=1}^N\psi(u_i(t))\quad\text{for $\lambda$-a.e.~$t\in S_P$}.
\end{align*}
Since $\{S_P:P\in \PP\}$ is a finite Borel partition of $[0,T]$, we get \eqref{eqpar}.
\end{proof}

\begin{proof}[Proof. of Theorem \ref{th:main}]

\noindent\textbf{The $\liminf$ inequality.}

Let $(\mu,\nnu)\in AC([0,T];\cP_1(\R^d)) \times \cM([0,T] \times \R^d; \UU )$ 
and $(\xx^N,\uu^N)\in AC([0,T];(\R^d)^N) \times L^1([0,T]; \UU^N)$, $N\in\N$,  
such that $\mu^N=\mu[\xx^N]\to \mu \text{ in } C([0,T];\cP_1(\R^d))$ and
$\nnu^N=\nnu[\xx^N,\uu^N] \rightharpoonup^* \nnu$ in $\cM([0,T] \times \R^d; \UU )$.

Since $L^N\geq0$, {\MF $L^N(x,\xx^N(t)) \to L(x,\mu_t)$ on compact sets, and $\mu[\xx^N] \rightharpoonup^* \mu$}, then for every compact $K\subset\R^d$ by \eqref{hp:Lconv} we have
\begin{equation}\label{eq:ab}
\begin{split}
\liminf_{N\to +\infty}&\int_{\R^d} L^N(x,\xx^N(t))\,\d\mu_t^N(x)\\
&\geq\liminf_{N\to +\infty}\int_{K} L^N(x,\xx^N(t))\,\d\mu^n_t(x) 
{\MF =} \int_{K} L(x,\mu_t)\,\d\mu_t(x).
\end{split}
\end{equation}
Since
$$\frac1N\sum^N_{i=1} L^N(x^N_i(t),\xx^N(t))=\int_{\R^d} L^N(x,\xx^N(t))\,\d\mu^N_t(x) 
$$
and $L\geq0$, by \eqref{eq:ab} we obtain
\begin{equation}\label{eq:a}
\liminf_{N\to \infty} \media_0^T  \frac 1N\sum^N_{i=1} L^N(x^N_i(t),\xx^N(t))\,\d t 
\geq \media_0^T\int_{\R^d} L(x,\mu_t)\,\d\mu_t(x)\,\d t.
\end{equation}
By \eqref{ineqpar} we have
\begin{equation}\label{eq:ineqcont}
	\frac{1}{N} \sum_{i=1}^N \psi(u^N_i(t)) \geq \Psi(\nnu^N_t | \mu^N_t) \qquad \mbox{for a.e. } t\in [0,T].
\end{equation}
and  Theorem \ref{t.funct:meas:lsc} yields
\begin{equation}\label{eq:b}
\liminf_{N\to \infty}  \media_0^T \Psi(\nnu^N_t | \mu^N_t)\,\dt
=\liminf_{N\to\infty}\Psi(\nnu^N |\tilde\mu^N) \geq  
\Psi(\nnu |\tilde\mu) =\media_0^T \Psi(\nnu_t | \mu_t)\,\dt.
\end{equation}
By \eqref{eq:a},  \eqref{eq:ineqcont}, and \eqref{eq:b} it follows \eqref{gammaliminf}.

\bigskip

\noindent\textbf{The $\limsup$ inequality.}
\nc
\GGG Recall that $\phi$ is an admissible function
satisfying \eqref{hp:phi}. \nc
Let $(\mu,\nnu)\in \AA$ such that $\cE(\mu,\nnu)<+\infty$
\GGG and $\int_{\R^d}\phi(|x|)\,\d\mu_0(x)<\infty$.

Since $\Psi(\nnu |\tilde\mu)<
+\infty$ we have \GGG $\nnu=\vv\tilde\mu$ 
for a Borel vector field 
$\vv:[0,T]\times\R^d\to\UU $. \nc
Since  $(\mu,\nnu)\in \AA$ the continuity equation
\begin{equation}\label{eq:vlasov2}
  \partial_t \mu_t+\nabla\cdot\big(\ww(t,\cdot)\mu_t\big)=0 
\end{equation}
holds with the vector field $\ww(t,x):=\ff(t,x)+\vv(t,x)$, $\ff(t,x):=\fF(x,\mu_t)$.
By \eqref{hp:linear} and Lemma \ref{p.first_moment} we have that
\begin{equation}
  \label{eq:57}
  \GGG
  \ff\in C([0,T]\times \R^d),\quad
  |\ff(t,x)|\le M(1+|x|),\quad
  \int_0^T\int_{\R^d}|\ww(t,x)|\,\d\mu_t(x)\,\d t <+\infty.
\end{equation}

By Theorem \ref{t:prob_repr} there exists a probability measure $\pi\in\cP(\Gamma_T)$ such that
$(e_t)_\#\pi=\mu_t$ for every $t\in[0,T]$ and it is concentrated on the absolutely continuous solutions of the ODE
\begin{equation}\label{char_eq}
\dot\gamma(t) = \ff(t,\gamma(t)) + \vv(t,\gamma(t)).
\end{equation}
The strategy of the proof consists in finding an appropriate sequence of measures $\pi^N\in\cP^N(\Gamma_T)$ 
narrowly convergent to $\pi$,
defining  $\mu^N_t:= (e_t)_\sharp \pi^N$ and $\xx^N$ a corresponding curve such that
$\mu[\xx^N]=\mu^N$. 
Then the objective is to construct
a suitable sequence of controls $\uu^N$
in such a way that
the sequence $(\xx^N,\uu^N)$ belongs to $\AA^N$, 
$\mu^N
\to \mu$  in $C([0,T];\cP_1(\R^d))$, $\nnu^N=\nnu[\xx^N,\uu^N]
\rightharpoonup^* \nnu$ in $\cM([0,T] \times \R^d; \UU )$ and
\eqref{limsup} holds.

\noindent STEP 1: (\textbf{Definition of auxiliary functionals.})

We define the set 
$$A:=\{\gamma\in\Gamma_T:\gamma \in AC([0,T];\R^d), \eqref{char_eq} \text{ holds for a.e. }t\in[0,T]\}$$
and we observe that $\pi(A)=1$.

Starting from $\mu$ and $L$ we define the functional $\cL:A\to [0,+\infty)$ by
\[\cL(\gamma) := \media_0^T L(\gamma(t),\mu_t)\,\d t.\]
Starting from $\psi$ and $\vv$ we define the functional $\cF:A\to [0,+\infty)$ by
\[\cF(\gamma) := \media_0^T \psi(\vv(t,\gamma(t)))\,\d t.\]
By Fubini's theorem and the finiteness of $\cE(\mu,\nnu)$ we have
\begin{equation}\label{eq:FL}
\media_0^T \int_{\R^d} L(x,\mu_t) \,\d\mu_t(x)\,\d t =\media_0^T \int_{A}L(e_t(\gamma),\mu_t) \,\d\pi(\gamma)\,\dt 
=  \int_{A} \cL(\gamma) \,\d\pi(\gamma)
\end{equation}
and
\begin{equation}\label{eq:Fpsi}
\media_0^T \int_{\R^d} \psi(\vv(t,x)) \,\d\mu_t(x)\,\d t =\media_0^T \int_{A} \psi(\vv(t, e_t(\gamma)) \,\d\pi(\gamma)\,\dt 
=  \int_{A} \cF(\gamma) \,\d\pi(\gamma).
\end{equation}

We define the functional  $\cH: A \to [0,+\infty)$ by
\[ \cH(\gamma) := \media_0^T \phi(|\vv(t, \gamma(t)|)\,\d t. \]
Starting by $\phi$ satisfying \eqref{hp:phi} we define the functional  $\cG: A \to [0,+\infty)$ by
\[\cG(\gamma) :=\phi(|\gamma(0)|) + \media_0^T \phi(|\gamma(t)|)\,\d t. \]
It is not difficult to show that $\cG$ and $\cL$ are continuous.
Here we prove that $\cF$ and $\cH$ are lower-semicontinuous.
Let $\gamma \in A$ and $(\gamma_k)_{k \in \N}$ be a sequence in $A$ such that 
$\lim_{k \to +\infty} \sup_{t \in [0,T]} |\gamma_k(t) - \gamma(t)| = 0$.
We define the sequence $f_k\in L^1([0,T];\R^d)$ by
$f_k(t):= \vv(t,\gamma_k(t))$. 

If $\sup_{k\in\N} \media_0^T \phi(|\vv(t, \gamma_k(t)|)\,\d t<+\infty$, 
then {\MF by de la Vall\'ee Poussin's criterion \cite[Proposition 1.12]{ambrosio2000functions} for equi-integrability and Dunford-Pettis theorem} there exist $g\in L^1([0,T];\R^d)$ and a subsequence (not relabeled) of $f_k$
weakly convergent in $L^1([0,T];\R^d)$ to $g$
such that
$$\liminf_{k\in\N} \media_0^T \phi(|\vv(t, \gamma_k(t)|)\,\d t \geq \media_0^T \phi(|g(t)|)\,\d t.$$
Since $\gamma_k$ satisfies
\begin{equation}\label{eq:inteql}
\gamma_k(t_2) - \gamma_k(t_1) = \int_{t_1}^{t_2} \left[ \vv(t,\gamma_k(t)) + \ff(t,\gamma_k(t)) \right] \d t, \qquad \forall \,  t_1,t_2 \in [0,T]
\end{equation} 
and  $\gamma$ satisfies
\begin{equation}\label{eq:inteq}
\gamma(t_2) - \gamma(t_1) = \int_{t_1}^{t_2} \left[ \vv(t,\gamma(t)) + \ff(t,\gamma(t)) \right] \d t, \qquad \forall \,  t_1,t_2 \in [0,T]
\end{equation} 
passing to the limit in \eqref{eq:inteql} as {\MF $k \to \infty$} we obtain
\begin{equation}
\gamma(t_2) - \gamma(t_1) = \int_{t_1}^{t_2} \left[ g(t) + \ff(t,\gamma(t)) \right] \d t.
\end{equation} 
By \eqref{eq:inteq} it holds
\begin{equation}
\int_{t_1}^{t_2} g(t) \d t = \int_{t_1}^{t_2} \vv(t,\gamma(t)) \d t, \qquad \forall t_1,t_2 \in [0,T],
\end{equation}  
and Lebesgue differentiation Theorem yields $g(t)=\vv(t,\gamma(t))$
for a.e. $t\in [0,T]$. 

\noindent STEP 2: (\textbf{Construction of $\pi^N$.})
We define the function $ \frF:= (\cF, \cL, \cG, \cH): A \to \R^4$. 

Notice that the finiteness of $\cE(\mu,\nnu)$, \eqref{eq:FL}, and \eqref{eq:Fpsi} 
imply that $\int_{A} \cF(\gamma)\,\d\pi(\gamma) < +\infty$, 
$\int_{A} \cL(\gamma)\,\d\pi(\gamma) < +\infty$ and
 $\int_{A} \cH(\gamma)\,\d\pi(\gamma) < +\infty$. 
Since  $\int_{A} \cG(\gamma)\,\d\pi(\gamma) = \int_0^T \int_{\R^d}\phi(|x|)\,\d\mu_t(x)\,\d t$,
 by Proposition \ref{t.phi_moment} we also have that $\int_{A} \cG(\gamma)\,\d\pi(\gamma) < +\infty$.

By Lusin's theorem applied to the space $A$ with the measure $\pi$ and the function $\frF$
 there exists a sequence of compact sets $A_k$ such that $A_k \subset A_{k+1} \subset A$, 
$\pi(A \setminus A_k) < \frac{1}{k}$, for all $k \geq 1$, and $\frF_{|A_k}$ is continuous. 
Moreover we have 
\begin{equation}\label{eq:esaustione}
\GGG \lim_{k\to\infty}\pi(A_k)=\nc \pi \Big( \bigcup_{j=1}^\infty A_j \Big) = 1,\quad
\GGG
\pi\Big(A\setminus \bigcup_{j=1}^\infty A_j\Big)=0.
\end{equation}
Then we define $\tilde \pi^k \in \cP(\Gamma_T)$ by
\begin{equation}
\tilde \pi^k := \frac{1}{\pi(A_k)} \pi\mres{A_k}.
\end{equation} 
\GGG It is easy to check that
$(\tilde\pi^k)_{k\in\N}$ weakly converges to $\pi$
as $k\to\infty$; 
since for each component $\frF_j$, $j=1,2,3,4$, of $\frF$ is nonnegative,
Beppo Levi monotone convergence Theorem yields \nc
\begin{equation}\label{eq:mon_Kn}
	\lim_{k \to +\infty} \int_{A_k} \frF_j(\gamma)\,\d\pi(\gamma) = \int_{\bigcup_{k=1}^\infty A_k} \frF_j(\gamma)\,\d\pi(\gamma) 
	= \int_A \frF_j(\gamma)\,\d\pi(\gamma),
\end{equation}
and \eqref{eq:esaustione} easily yields
\begin{equation}\label{eq:limit1}
\lim_{k \to \infty}\left|\int_{\Gamma_T} \frF(\gamma)\,\d\tilde\pi^k(\gamma) -
  \int_{\Gamma_T} \frF(\gamma)\,\d\pi(\gamma)\right|=0.
\end{equation} 
\GGG 
Since $A_k$ is compact, 
we can find a sequence of atomic measures
$$m \mapsto \tilde\pi^k_m:=\frac{1}{m}\sum_{i=1}^m \delta_{\gamma_{i,k,m}},\quad
\gamma_{i,k,m}\in A_k,$$
\nc narrowly convergent to $\tilde\pi^k$ as $m\to+\infty$.
Since $\frF_{|A_k}$ is bounded and continuous, in particular it holds that
\begin{equation}
\lim_{m \to \infty}\int_{\Gamma_T}\frF(\gamma)\,\d\tilde\pi^k_m(\gamma)  = \int_{\Gamma_T} \frF(\gamma) \,\d\tilde\pi^k(\gamma).
\end{equation}
Hence, for every $k\in\N$ there exists $\bar{m}(k)$ satisfying
\begin{equation} \label{eq:limit2} 
W(\tilde\pi_m^k,\tilde\pi^k)\le \frac 1k \quad \mbox{and} \quad
\left| \int_{\Gamma_T}\frF(\gamma)\,\d\tilde\pi^k_m(\gamma) - 
\int_{\Gamma_T} \frF(\gamma) \,\d\tilde\pi^k(\gamma) \right| \leq \frac 1k, \qquad \forall \, m \geq \bar{m}(k),
\end{equation} 
where $W$ is any distance metrizing the weak convergence.

We define 
$\bar\pi^k:= \tilde\pi^k_{\bar{m}(k)}$ and we clearly have that
$\bar\pi^k\in\cP^{\bar m(k)}(\Gamma_T)$,
 $W(\bar\pi^k,\pi)\to 0$
as $k\to\infty$ and, by \eqref{eq:limit1} and  \eqref{eq:limit2}, 
\begin{equation}
\label{eq:3}
\lim_{k\to\infty}\left| \int_{\Gamma_T}\frF(\gamma)\,\d\bar\pi^k(\gamma) -
  \int_{\Gamma_T} \frF(\gamma) \,\d\pi(\gamma) \right| =0.
\end{equation}
Since we can choose the sequence $k\mapsto \bar{m}(k)$ strictly increasing, we can 
consider the sequence $N\mapsto \pi^N$ such that
\GGG
$\pi^N\in\cP^{N}(\Gamma_T)$, $\pi^N:=\bar\pi^k$ when $\bar m(k)\le N<\bar m(k+1)$; \nc
$\pi^N$ narrowly converges to $\pi$ as $N\to+\infty$
  and
\begin{equation}
\label{eq:3bis}
\lim_{N\to+\infty}\int_{\Gamma_T}\frF(\gamma)\,\d\pi^N(\gamma) =
  \int_{\Gamma_T} \frF(\gamma) \,\d\pi(\gamma). 
\end{equation}

Since all the components of $\frF$ are nonnegative and lower semicontinuous maps, {\MF by a combination of \cite[Proposition 1.62 (a)]{ambrosio2000functions} and \cite[Proposition 1.80]{ambrosio2000functions} we have that}  \eqref{eq:3}
yields in particular that the measures 
\begin{displaymath}
  \sigma^N_1:=\cF\,\pi+\cF\,\pi^N,\quad
  \sigma^N_2:=\cG\,\pi+\cG\,\pi^N, \quad
  \sigma^N_3:=\cH\,\pi+\cH\,\pi^N, \quad
   \sigma^N_4:=\cL\,\pi+\cL\,\pi^N
\end{displaymath}
weakly converge to $\sigma_1:=2\mathcal F\,\pi$, $\sigma_2:=2\mathcal G\,\pi$, 
$\sigma_3:=2\mathcal H\,\pi$ and $\sigma_4:=2\mathcal L\,\pi$ respectively. 
In particular they are uniformly tight, so that for
every $\eps>0$ there exists $
\bar N(\eps)\in \N$ and a compact set $B_\eps$ and such that
\begin{equation}
  \label{eq:5}
  B_\eps\subset A_N,\quad 
  (\pi+\pi^N)(\Gamma_T\setminus B_\eps)+\int_{\Gamma_T\setminus B_\eps} \Big(\mathcal F+ \cL + \mathcal
  H + \cG \Big)\,\d(\pi+\pi^N)\le \eps\quad
  \text{for every }N\ge \bar N(\eps).
\end{equation}

\noindent STEP 3: (\textbf{Definition of $(\xx^N,\uu^N)$ and convergence.})
We define $\mu_t^N:= (e_t)_\sharp \pi^N\in\cP^{N}(\R^d)$ 
and we denote by $\xx^N$ a corresponding curve such that
$\mu[\xx^N]=\mu^N$. 
We define
 \begin{align}\label{f^N_i}
   \GGG
   \ff^N(t,x):={}&\fF^N(x,\xx^N(t))=\fF(x,\mu^N_t),
   \quad
   \vv^N(t,x):= \vv(t,x) + \ff(t,x) - \ff^N(t,x)\\
   u^N_i(t):={}& \vv^N(t,x^N_i(t)) 
\end{align} 
and $\uu^N=(u_1^N,\ldots,u_N^N)$.
\GGG Notice that
\begin{displaymath}
  \text{$\ff(t,x) - \ff^N(t,x)\in \UU$ and 
    $\uu^N\in \UU^N$ thanks to the compatibility condition \eqref{eq:65}.}
\end{displaymath}
\nc
We have that $\nnu^N:=\nnu[\xx^N,\uu^N]=\vv^N\mu^N$.
Since each component $x_i^N$ of $\xx^N$ belongs to $A$, then
the sequence $(\xx^N,\uu^N)$ belongs to $\AA^N$, so that 
$(\mu^N,\nnu^N)\in\AA$.
Using the same computation of the proof of Proposition \ref{t.phi_moment}, taking into account that $\mu^N$
satisfies 
\begin{equation}\label{eq:CEA}
  \partial_t \mu^N_t+\nabla\cdot\Big(\big(\ff(t,\cdot)+\vv(t,\cdot)\big)\mu^N_t\Big)=0,
\end{equation}
with (recall \eqref{eq:Fpsi} and \eqref{eq:3bis})
\begin{displaymath}
  \media_0^T\int_{\R^d} \psi(\vv(t,x))\,\d\mu^N_t(x)\,\d t
  =\int_{\Gamma_T} \cF(\gamma)\,\d\pi^N(\gamma)\le 1+\cE(\mu|\nnu)
\end{displaymath}
for $N$ sufficiently big, 
we obtain by \eqref{eq:57}, \eqref{eq:58} and \eqref{eq:56}
that
\begin{equation}\label{momNbis}
\sup_{N\in \N} \sup_{t \in [0,T]} \int_{\R^d} |x|\,\d\mu_t^N(x)<+\infty,
\quad
\sup_{N\in \N} \sup_{t \in [0,T]} \int_{\R^d} \phi(|x|)\,\d\mu_t^N(x) < +\infty,
\end{equation}
which implies the uniform convergence
$\mu^N \to \mu$  in $C([0,T];\cP_1(\R^d))$ 
\GGG and
the uniform estimate
\begin{equation}
  \label{eq:60}
  |\ff^N(t,x)|\le M'(1+|x|)\quad\text{for every }t\in [0,T],\ x\in \R^d,\ N\in \N,
\end{equation}
for a suitable constant $M'>0$. \nc
By a direct computation, using the assumption \eqref{hp:Fconv}, we obtain that
 $\nnu^N \rightharpoonup^* \nnu$ in $\cM([0,T] \times \R^d; \UU )$.\\

\noindent STEP 4: (\textbf{Definition and convergence of $\cF^N$.}) 
We define $\cF^N:A\to [0,+\infty)$ by
\begin{equation}\label{def:F^k}
\cF^N(\gamma) := \media_0^T \psi(\vv^N(t,\gamma(t)))\,\d t.
\end{equation}

Here we show that the sequence
$\cF^N$ converges to $\cF$ uniformly on every compact set $\Lambda \subset A_h$ for some $h\in\N$. 
To do it, we fix $\Lambda \subset A_h$ and we prove that for any $\gamma\in\Lambda$ and
 every sequence $(\gamma_N)_{N \in \N} \subset \Lambda$ such that $\sup_{t \in [0,T]} |\gamma_N(t) - \gamma(t)| \to 0$, we have
$\cF^N(\gamma_N) \to \cF(\gamma)$ as $N\to+\infty$.

By the assumption \eqref{hp:Fconv} we have that
\begin{equation}
  \lim_{N\to+\infty}|\ff(t,\gamma_N(t)) - \ff^N(t,\gamma_N(t))| = 0, \qquad \forall\, t \in [0,T].
\end{equation}
Since $\cH$ is continuous in $A_h$ it holds
\begin{equation}\label{contphiN}
\lim_{N\to+\infty}\int_0^T \phi(|\vv(t,\gamma_N(t))|)\, \d t = \int_0^T \phi(|\vv(t,\gamma(t))|)\, \d t.
\end{equation} 
Since $\phi$ is strictly convex and superlinear, by Visintin's Theorem \cite[Thm.~3]{visintin1984strong}
$\vv(\cdot,\gamma_N(\cdot))$   strongly converges in $L^1(0,T)$ to $\vv(\cdot,\gamma(\cdot))$. 
Then, using also the continuity of $\psi$, along a subsequence (still denoted by $\gamma_N$)  we have
\begin{equation}
	\lim_{N\to+\infty}\psi(\vv(t,\gamma_N(t)) + \ff(t,\gamma_N(t)) - \ff^N(t,\gamma_N(t)) = \psi(\vv(t,\gamma(t))), 
	\qquad \text{for a.e. } t \in [0,T].
\end{equation}
\GGG
Since by \eqref{eq:56} and \eqref{eq:60} we have
\begin{equation}
  |{\ff(t,\gamma_N(t)) }- \ff^N(t,\gamma_N(t))|
  \leq  (M+M')\big(1+|\gamma_N(t)|\big)
\end{equation}
\nc 
then, using the doubling property and the uniform convergence
of $\gamma_N$, we can find a constant $C$ such that
\begin{equation}
\psi\Big(\vv(t,\gamma_N(t)) + \ff(t,\gamma_N(t)) - \ff^N(t,\gamma_N(t))\Big)
\leq  C\Big(1+ \phi(|\vv(t,\gamma_N(t))|)\Big).
\end{equation}
By \eqref{contphiN} the generalized dominated convergence theorem (see for instance {\GGG\cite[Thm.~4, page 21]{evans2015measure}})
shows that
\begin{equation}
	\lim_{N\to+\infty}\int_0^T\psi(\vv^N(t,\gamma_N(t))\,\d t  = \int_0^T \psi(\vv(t,\gamma(t)))\,\d t.
\end{equation}

\noindent STEP 5: (\textbf{Definition and convergence of $\cL^N$.}) 
We define $\cL^N:A\to [0,+\infty)$ by
\begin{equation}\label{def:L^k}
\cL^N(\gamma) := \media_0^T L^N(\gamma(t),\xx^N(t))\,\d t.
\end{equation}

Here we show that the sequence
$\cL^N$ converges to $\cL$ uniformly on every compact set $\Lambda \subset A_h$ for some $h\in\N$. 
As in step 4, we fix $\Lambda \subset A_h$ and we prove that  
for every sequence $(\gamma_N)_{N \in \N} \subset \Lambda$, with $\sup_{t \in [0,T]} |\gamma_N(t) - \gamma(t)| \to 0$, we have
$\cL^N(\gamma_N) \to \cL(\gamma)$ as $N\to+\infty$.
Indeed, by \eqref{hp:Lconv},
\begin{equation}
	\lim_{N\to+\infty}L^N(\gamma_N(t),\xx^N(t)) = L(\gamma(t),\mu_t), \qquad \forall\, t \in [0,T].
\end{equation}
Since $(\gamma_N)_N$ is bounded and $\mu^N\to \mu$  in $C([0,T];\cP_1(\R^d))$,
by \eqref{hp:Lconv}
we obtain that
$$ \sup_{N\in\N} L^N(\gamma_N(t),\xx^N(t))<+\infty.$$
By dominated convergence we conclude.

\noindent STEP 6: (\textbf{Conclusion.}) 
By the growth assumption \eqref{eq:57} and \eqref{eq:60} on $\ff,\ff^N$, the doubling property of $\phi$ and \eqref{hp:phi} we have
\begin{equation}\label{lastest}
\cF^N(\gamma) \leq C(1+ \cF(\gamma) + \cG(\gamma)) \qquad \forall\, \gamma\in A, \quad \forall \, N\in \N.
\end{equation}
Moreover, by \eqref{hp:Lconv} and the uniform convergence of $\mu^N$ to $\mu$, there exists a constant $C$ such that
\begin{equation}\label{lastest2}
\cL^N(\gamma) \leq  \cL(\gamma) +C \qquad \forall\, \gamma\in A, \quad \forall \, N\in \N.
\end{equation}
Fix $\varepsilon > 0$ and let $B_\varepsilon$ and $\bar{N}(\varepsilon)$
such that  \eqref{eq:5} holds and
\begin{equation}
 \left| \int_{\Gamma_T} \cF(\gamma)\d\pi^N(\gamma) - \int_{\Gamma_T} \cF(\gamma)\d\pi(\gamma) \right|<\eps. 
\end{equation}
By \eqref{lastest}, \eqref{lastest2} and \eqref{eq:5} we have
\begin{equation}
\int_{\Gamma_T \setminus B_\varepsilon} \cF^N(\gamma) \d(\pi + \pi^N)(\gamma) \leq \varepsilon ,\quad
\int_{\Gamma_T \setminus B_\varepsilon} \cL^N(\gamma) \d(\pi + \pi^N)(\gamma) \leq \varepsilon 
\qquad \forall\, N \geq \bar{N}(\varepsilon).
\end{equation}
Moreover, from the previous step, there exists $\tilde{N}(\varepsilon)$ such that
\begin{equation}
 \sup_{\gamma\in B_\eps} |\cF^N(\gamma)-\cF(\gamma)| \leq \varepsilon ,
 \quad  \sup_{\gamma\in B_\eps} |\cL^N(\gamma)-\cL(\gamma)| \leq \varepsilon 
\qquad \forall\, N \geq \tilde{N}(\varepsilon).
\end{equation}
Hence
\begin{equation}
\begin{split}
\left| \int_{\Gamma_T} \cF^N(\gamma)\d\pi^N(\gamma) - \int_{\Gamma_T} \cF(\gamma)\d\pi(\gamma) \right| 
&\leq \left| \int_{B_\varepsilon} \cF^N(\gamma)\d\pi^N(\gamma) - \int_{B_\varepsilon} \cF(\gamma)\d\pi(\gamma) \right| \\
&+ \left| \int_{\Gamma_T \setminus B_\varepsilon} \cF^N(\gamma)\d\pi^N(\gamma) - \int_{\Gamma_T \setminus B_\varepsilon} \cF(\gamma)\d\pi(\gamma) \right| \\
&\leq \varepsilon + 2\varepsilon, \qquad \forall \, N \geq \max\lbrace{\bar{N}(\varepsilon),\tilde{N}(\varepsilon)\rbrace},
\end{split}
\end{equation}
which shows that
\[ \lim_{N\to \infty} \int_{\Gamma_T} \cF^N(\gamma)\d\pi^N(\gamma) = \int_{\Gamma_T} \cF(\gamma)\d\pi(\gamma).\]
Analogously we obtain
\[ \lim_{N\to \infty} \int_{\Gamma_T} \cL^N(\gamma)\d\pi^N(\gamma) = \int_{\Gamma_T} \cL(\gamma)\d\pi(\gamma).\]

\end{proof}

\subsection{Convergence of minima.}

\GGG
\noindent\textbf{Proof of Theorem \ref{thm:compactness}.}

\noindent
(\textbf{Equicontinuity}) 
\nc
Let $N$ be fixed and $s \leq t$. From the constraint \eqref{eq:state} we get
\begin{equation}
\begin{split}
W_1(\mu^N_s,\mu^N_t) & \leq \frac{1}{N} \sum_{i=1}^N |x_i^N(s)-x_i^N(t)| \\
&\leq \frac{1}{N} \sum_{i=1}^N \int_s^t |\fF^N(x_i(r),\xx^N(r))|\,\d r +  \frac{1}{N} \sum_{i=1}^N \int_s^t |u^N_i(r)|\,\d r \\
&\leq \tilde C(t-s) + \int_s^t \frac{1}{N} \sum_{i=1}^N |u^N_i(r)|\,\d r,
\end{split}
\end{equation}
where $\tilde C:= A + 2B\left( \sup_N|\xx^N(0)|_N + AT + \sup_N \int_0^T |\uu^N(r)|_N\,\d r \right)e^{2BT}$ 
(see the proof of \eqref{equicont}) which is 
uniformly bounded, since $\mu^n_0$ is converging in $\cP_1(\R^d)$ 
and $\cE(\mu^N,\nnu^N)$ is uniformly bounded.

\GGG
By Remark \ref{rem:moment-reinforcement},
we can select an admissible function $\theta$ satisfying \eqref{eq:42} 
with $\cK:=\{\mu_0\}\cup\{\mu^N_0:N\in\N\}$. 
The uniform bound on $\cE(\mu^N,\nnu^N)$ \nc implies that
$$ \sup_N \int_0^T \frac{1}{N} \sum_{i=1}^N \theta (|u^N_i(r)|) \,\d r<+\infty,$$
by the convexity and superlinearity of $\theta$ there exists a uniform modulus of continuity $\omega: [0,+\infty)\to[0,+\infty)$
such that $ \sup_N \int_s^t \frac{1}{N} \sum_{i=1}^N |u^N_i(r)|\,\d r\leq \omega (t-s)$.

Hence we have just shown the equicontinuity property 
\[W_1(\mu^N_s,\mu^N_t) \leq \omega(|t-s|) +\tilde C|t-s| \qquad \forall\, t,s\in[0,T].\]
\textbf{(Compactness)} From Theorem \ref{t.phi_moment} we have
\begin{equation}
\sup_{N\in \N} \sup_{t \in [0,T]} \int_{\R^d} \theta(|x|)\,\d\mu^n_t(x) < +\infty.
\end{equation}
This implies that the family $(\mu^N)_{N \in \N} \subset \cP_1(\R^d)$ is relatively compact, 
(see e.g. \cite[Prop.~7.1.5]{Ambrosio2008gradient}). 

The application of Ascoli-Arzel\`a Theorem provides a limit curve $\mu \in C([0,T];\cP_1(\R^d))$ and a subsequence, 
still denoted by $\mu^N$, such that
\begin{equation}\label{compconv}
 \sup_{t \in [0,T]}W_1(\mu[\xx^N]_t,\mu_t) \to 0. 
 \end{equation}
Concerning the control part we write $\nnu^N = \vv^N\mu^N$. 
Since $\cE^N(\xx^N,\uu^N)$ is uniformly bounded, we have 
\[\sup_{N\in \N}\int_0^T\int_{\R^d}\psi(\vv^N(t,x))\,\d\mu^N_t(x) \,\d t <+\infty.\]
By the superlinearity of $\psi$ and the convergence \eqref{compconv}, using the same argument 
of the proof of \cite[Th.~5.4.4]{Ambrosio2008gradient}
we obtain that there exist $\vv:[0,T]\times\R^d\to\UU $ 
and a subsequence (again denoted by $\vv^N$)
such that
\[\int_0^T\int_{\R^d}\psi(\vv(t,x))\,\d\mu_{t}(x) \,\d t <+\infty\]
and
\[ \lim_{N \to \infty} \int_0^T\int_{\R^d} \xxi(t,x)\cdot \vv^N(t,x) \d\mu^N_t(x)\,\d t 
= \int_0^T\int_{\R^d} \xxi(t,x)\cdot \vv(t,x) \d\mu_t(x)\,\d t, \qquad \forall \, \xxi \in C_c^\infty([0,T] \times \R^d;\R^d).\]
This proves the convergence of $\nnu^N \to \nnu:=\vv\mu$ in $\cM([0,T] \times \R^d;\UU )$
{\GGG and the fact that $(\mu,\nnu)$ satisfy \eqref{eq:vlasov}. \quad$\square$

\bigskip
\noindent
\textbf{Proof of Theorem \ref{t:minima}}

The first two claims are standard consequence of 
the $\Gamma$-convergence result Theorem \ref{th:main} 
and the coercivity property stated in Theorem \ref{thm:compactness}.
We thus consider the third claim.

Let us fix $\mu_0\in \cP_1(\R^d)$ with compact support and $(\mu,\nnu)\in P(\mu_0)$. 
By Theorem \ref{th:main} we can find a sequence of discrete 
solutions $(\hat\xx^N,\hat\uu^N)$ corresponding to initial data 
$\hat\xx_0^N$ supported in $\supp(\mu_0)$ and measures
$(\hat\mu^N,\hat\nnu^N)$ converging to $(\mu,\nnu)$ 
such that \eqref{eq:63} and \eqref{limsup} holds.
Theorem \ref{th:main} also yields
$\lim_{N\to\infty}E^N(\hat\xx^N_0)=E(\mu_0)$.

Let now $(\xx_0^N)_{N\in \N}$ be any other sequence satisfying \eqref{eq:49}
with $(\xx^N,\uu^N)\in P(\xx_0^N)$ and $\mu^N=\mu[\xx^N]$, $\nnu^N=\nnu[\xx^n,\uu^N]$.
Applying Lemma \ref{le:convergence} we deduce that the associated measures
$\mu^N_0$ satisfy
\begin{displaymath}
  \lim_{N\to\infty}\calC_\phi(\hat\mu^N_0,\mu^N_0)=0.
\end{displaymath}
Up to a permutation of the initial points $(\hat x^N_{0,1},\hat x^N_{0,2},\cdots,\hat x^N_{0,N})$
(and of the corresponding solutions $(\hat \xx^N,\hat \uu^N)$)
which however leaves $\hat\mu_0^N,\hat\mu^N,\hat\nnu^N$ invariant,
we may assume by \eqref{eq:36} that 
\begin{equation}
  \label{eq:66}
  c_N=\calC_\phi(\hat\mu^N_0,\mu^N_0)=
  \frac 1N\sum_{i=1}^N\phi(|\hat x^N_{0,i}- x^N_{0,i}|).
\end{equation}
For $0 <\delta <T$ and $\yy^{N,\delta}:= \delta^{-1}(\hat\xx^N_0 -\xx^N_0)$ 
we can then define a new competitor by
\begin{align*}
  \xx^{N,\delta}(t):={}&
  \begin{cases}
    (1-t/\delta)\xx^N_0+t/\delta\, \hat\xx^N_0&\text{if }t\in [0,\delta),\\
    \hat\xx^N(t-\delta)&\text{if }t\in [\delta,T],
  \end{cases}\\
  \uu_i^{N,\delta}(t):={}&
  \begin{cases}
    \yy^{N,\delta}-\fF^N(x_i^{N,\delta},\xx^{N,\delta}(t))&\text{if }t\in [0,\delta),\\
    \hat\uu^N(t-\delta)&\text{if }t\in [\delta,T].
  \end{cases}
\end{align*}
It is easy to check that $(\xx^{N,\delta},\uu^{N,\delta})\in \AA(\xx_0^N)$ 
so that 
$E^N(\xx_0^N)\le \cE^N(\xx^{N,\delta},\uu^{N,\delta})$. 
On the other hand
\begin{align*}
  T\cE^N(\xx^{N,\delta},\uu^{N,\delta})
  &\le 
    \frac 1N\int_0^\delta \sum_{i=1}^N L^N(x_i^{N,\delta}(t),\xx^{N,\delta}(t))\,\d t+
    \frac 1 N \int_0^\delta\sum_{i=1}^N \psi\big( \yy^{N,\delta} - \fF^N(x_i^{N,\delta},\xx^{N,\delta}(t)) \big)\, \d t\\
    &+T\cE^N(\hat\xx^N,\hat\uu^N).
\end{align*}
From the doubling property and the compactness of supports of $(\xx^N_0)$, applying the same argument as in the proof of Theorem \ref{th:main} we get
\begin{align*}
  \psi\big( \yy^{N,\delta} - \fF^N(x_i^{N,\delta},\xx^{N,\delta}(t)) \big) &\le C(1+\phi(|\hat\xx^N_0 -\xx^N_0|/\delta)) \\
  &\le C\mathrm e^{K/\delta}(1+\phi(|\hat\xx^N_0 -\xx^N_0|)) \quad 0<\delta<1.
\end{align*}
Setting $\mu^{N,\delta}_t=\mu[\xx^{N,\delta}(t)]$ we get,
\begin{equation}
  \label{eq:72}
  T\big(\cE^N(\xx^{N,\delta},\uu^{N,\delta})-\cE^N(\hat\xx^N,\hat\uu^N)\big)
  \le C\,c_N\delta (1+\mathrm e^{K/\delta})
    +\delta \sup_{t\in [0,1]}\int_{\R^d}L^N(x,\mu^{N,1}_t)\,\d\mu^{N,1}_t.
\end{equation}
If we choose $\delta=\delta(N):= -K\big(\log(c_N)\big)^{-1}$, since 
$\lim_{N\to\infty}\sup_{t\in [0,1]}W_1(\mu^{N,1}_t,\mu_0)=0$, 
we see that the right hand side of \eqref{eq:72} tends to $0$ as $N\to\infty$, so
that we eventually obtain
\begin{displaymath}
  \limsup_{N\to\infty} E^N(\xx_0^N)\le 
  \limsup_{N\to\infty}\cE^N(\xx^{N,\delta},\uu^{N,\delta})\le 
  \limsup_{N\to\infty}\cE^N(\hat\xx^N,\hat\uu^N) =E(\mu_0).
\end{displaymath}}

\section*{Acknowledgements}
{\GGG We wish to thank Filippo Santambrogio for useful discussions concerning the third claim of Theorem 
\ref{t:minima}.}

{\MF Massimo Fornasier  acknowledges
 the financial support provided by the ERC-Starting Grant ``High-Dimensional Sparse Optimal Control'' (HDSPCONTR) 
 and the DFG-Project FO 767/7-1 ``Identification of Energies from the Observsation of Evolutions''.}
Giuseppe Savar\'e acknowledges the financial support provided by 
Cariplo foundation and Regione Lombardia via project 
``Variational evolution problems and optimal transport''.
Carlo Orrieri acknowledges the financial support provided by PRIN 20155PAWZB ``Large Scale Random Structures''.

\bibliography{biblio}
\bibliographystyle{plain}

\end{document}